\documentclass[11pt]{article}
\usepackage[margin=1in]{geometry}

\input xy
\xyoption{all}

\usepackage{amsmath,amsfonts,amsthm,amssymb,latexsym,amscd,amsopn,graphics,color,enumerate,lscape,enumitem}
\usepackage[colorlinks]{hyperref}
\usepackage{tikz-cd}

\allowdisplaybreaks

\theoremstyle{plain}
  \newtheorem{theorem}{Theorem}
  \newtheorem{corollary}[theorem]{Corollary}
  \newtheorem{thm}{Theorem}[section]
  \newtheorem{conj}[theorem]{Conjecture}
  \newtheorem{prop}[thm]{Proposition}
  \newtheorem{cor}[thm]{Corollary}
  \newtheorem{lemma}[thm]{Lemma}
  
\theoremstyle{definition}
  \newtheorem{defn}[thm]{Definition}
  \newtheorem{example}[thm]{Example}
  
  \newtheorem{remark}[thm]{Remark}
  
\theoremstyle{remark}

\DeclareMathOperator{\Cl}{Cl}

\DeclareMathOperator{\Frac}{Frac}

\DeclareMathOperator{\Hom}{Hom}

\DeclareMathOperator{\Aut}{Aut}

\DeclareMathOperator{\Sym}{Sym}

\DeclareMathOperator{\Norm}{N}

\DeclareMathOperator{\Stab}{Stab}
\def\Z{{\mathbb Z}}
\def\T{{\mathcal T}}
\def\W{{\mathcal W}}
\def\bZ{{\mathbb Z}}
\def\irr{{\rm irr}}
\def\nostab{{\rm nostab}}

\def\GL{{\rm GL}}
\def\SL{{\rm SL}}
\def\Cl{{\rm Cl}}

\def\Sym{{\rm Sym}}
\def\Det{{\rm Det}}
\def\Aut{{\rm Aut}}
\def\Stab{{\rm Stab}}

\def\bad{{\rm bad}}

\def\Avg{{\rm Avg}}

\def\P{{\mathbb P}}
\def\Disc{{\rm Disc}}

\def\proj{{\rm proj}}
\def\red{{\rm red}}
\def\Vol{{\rm Vol}}
\def\prim{{\rm prim}}
\def\R{{\mathbb R}}
\def\C{{\mathbb C}}
\def\F{{\mathbb F}}
\def\bR{{\mathbb R}}

\def\FF{{\mathcal F}}
\def\RR{{\mathcal R}}

\def\Q{{\mathbb Q}}
\def\H{{\mathcal H}}
\def\U{{\mathcal U}}
\def\V{{\mathcal V}}
\def\Z{{\mathbb Z}}
\def\P{{\mathbb P}}
\def\F{{\mathbb F}}
\def\Q{{\mathbb Q}}

\def\cI{{\mathcal I}}
\def\bQ{{\mathbb Q}}
\def\cC{{\mathcal C}}

\def\bZ{{\mathbb Z}}
\def\bP{{\mathbb P}}

\def\bQ{{\mathbb Q}}
\def\cO{{\mathcal O}}
\def\cI{{\mathcal I}}

\def\dh{\;\mathrm{d}h}
\def\dd{\;\mathrm{d}}
\def\sgn{{\mathrm{sgn}}}
\def\fR{{\mathfrak R}}
\def\If{I_f}
\def\II{{I}}
\def\tt{{\tau}}
\def\rmH{\mathrm{H}}

\newcommand{\var}{\mathrm{Var}}
\newcommand{\SO}{\mathrm{SO}}

\newcommand{\PP}{\mathbb{P}}

\newcommand{\Zp}{\mathbb{Z}_p}
\newcommand{\Fp}{\mathbb{F}_p}

\newcommand{\Si}{\mathfrak{S}}

\newcommand{\J}{\mathcal{J}}

\title{Odd degree number fields with odd class number}

\author{Wei Ho, Arul Shankar, and Ila Varma}

\date{}

\begin{document}

\maketitle

\begin{abstract}
For every odd integer $n \geq 3$, we prove that there exist infinitely
many number fields of degree $n$ and associated Galois group $S_n$
whose class number is odd. To do so, we study the class groups of
families of number fields of degree $n$ whose rings of integers arise
as the coordinate rings of the subschemes of $\bP^1$ cut out by
integral binary $n$-ic forms. By obtaining upper bounds on the mean
number of $2$-torsion elements in the class groups of fields in these
families, we prove that a positive proportion (tending to $1$ as $n$
tends to $\infty$) of such fields have trivial $2$-torsion subgroup in
their class groups and narrow class groups. Conditional on a tail
estimate, we also prove the corresponding lower bounds and obtain the
exact values of these averages, which are consistent with the
heuristics of Cohen--Lenstra--Martinet--Malle and Dummit--Voight.

Additionally, for any order $\cO_f$ of degree $n$ arising from an
integral binary $n$-ic form $f$, we compare the sizes of
$\Cl_2(\cO_f)$, the $2$-torsion subgroup of ideal classes in $\cO_f$,
and $\cI_2(\cO_f)$, the $2$-torsion subgroup of ideals in $\cO_f$. For
the family of orders arising from integral binary $n$-ic forms and
contained in fields with fixed signature $(r_1,r_2)$, we prove that
the mean value of the difference $|\Cl_2(\cO_f)| -
{2^{1-r_1-r_2}}|\cI_2(\cO_f)|$ is equal to $1$, generalizing a result of
Bhargava and the third-named author for cubic fields.  Conditional on
certain tail estimates, we also prove that the mean value of
$|\Cl_2(\cO_f)| - {2^{1-r_1-r_2}}|\cI_2(\cO_f)|$ remains $1$ for
certain families obtained by imposing local splitting and maximality
conditions.
\end{abstract}

\section{Introduction}
\enlargethispage{\baselineskip}

The Cohen-Lenstra heuristics \cite{cohenlenstra} give precise
predictions for the distribution of ideal class groups in families of
quadratic fields. Very few cases of these conjectures have been
proved; among them are the celebrated results of Davenport--Heilbronn
\cite{davenportheilbronn2} on the average number of $3$-torsion
elements in the class groups of quadratic fields,
and of Fouvry--Kluners \cite{FK} on the $4$-ranks of the class
groups of quadratic fields. These heuristics were generalized by
Cohen--Martinet \cite{cohenmartinet} to describe the distribution of
ideal class groups in families of number fields of
fixed degree over a fixed base field. In 2010, Malle \cite{clnum2}
proposed a modification
of Cohen--Martinet's heuristics to account for observed variations in
the asymptotic behavior of the $p$-part of the class groups of
families over a base field containing the $p$th roots of unity; for
example, for $p = 2$ and odd $n$, the modified heuristics yield the following
predictions on the mean number of $2$-torsion ideal classes in degree $n$ $S_n$-number fields over $\bQ$ with signature $(r_1,r_2)$, i.e., number fields with $r_1$ real embeddings and $r_2$ pairs of conjugate complex embeddings, and whose normal closure over $\bQ$ has Galois group $S_n$.

\begin{conj}[Cohen-Lenstra-Martinet-Malle]  \label{conj:malle}
Fix an odd integer $n \geq 3$ and a pair of nonnegative integers
$(r_1,r_2)$ such that $r_1 + 2r_2 = n$. Consider the set of
isomorphism classes of degree $n$ $S_n$-number fields with signature $(r_1,r_2)$. The average number
of $2$-torsion elements in the ideal class groups of such fields is
	\begin{equation} \label{eq:avgvalue}
	1 + 2^{1-r_1 - r_2}
	\end{equation}
when these fields are ordered by discriminant.
\end{conj}

The only proven cases of the above conjecture are when $n = 3$, due to
Bhargava \cite{manjulcountquartic}. In this paper, we provide evidence
toward {\em all} cases of Conjecture \ref{conj:malle} by computing
the average size of the $2$-torsion subgroups of ideal class groups of
certain infinite families of number fields of fixed odd degree $n$;
even though we do not average over the family of all number fields of
a given signature ordered by discriminant, the mean values coincide
with \eqref{eq:avgvalue}, conditional on a certain tail
estimate. Unconditionally, we prove that an infinite number of odd degree $n$ $S_n$-fields with signature $(r_1,r_2)$ have odd class number. We also compute the
average size of the $2$-torsion subgroup of the {\em narrow} class
groups of the same infinite families, which allows us to give
analogues of the Cohen-Lenstra-Martinet-Malle heuristics predicting
the asymptotic behavior of the narrow class groups in families of
number fields of fixed odd degree and signature.

\medskip

In order to state our results more precisely, we first describe the
families of number fields we study, which arise from families of integral binary
$n$-ic forms.  Given an integer $n \geq 3$, to a nonzero integral binary
$n$-ic form $f \in \Sym_n(\Z^2)$, we may naturally associate the
coordinate ring $R_f$ of the subscheme of ${\bP}^1_\Z$ cut out by $f$
(see Nakagawa \cite{nakagawa} and Wood \cite{rings}).  Define the
family $\mathfrak{R}_H$ to be the multiset of rings
	$$\mathfrak{R}_H = \{R_f \mid f \in \Sym_n(\bZ^2)\}.$$
There is a {\em height} ordering on $\mathfrak{R}_H$ arising from
the height ordering $H$ on $\Sym_n(\Z^2)$, where $H(f)$ is
defined as the maximum absolute value of the coefficients of $f$. Note that 
although two rings in $\fR_H$ may be
isomorphic, their heights need not be equal. For example, if $\gamma
\in \SL_2(\bZ)$, and we define the action $\gamma f(x,y) :=
f((x,y)\gamma)$ on the space of integral binary $n$-ic forms, then it
is always true that $R_f \cong R_{\gamma f}$, but it is not in general
true that $H(f) = H(\gamma f)$. Nevertheless, there is a well-defined
isomorphism class of rings $R_{[f]}$ associated to an
$\SL_2(\bZ)$-orbit $[f] \in \SL_2(\bZ) \backslash \Sym_n(\bZ^2)$ since
$R_{[f]}$ is isomorphic to $R_g$ if and only if $g = \gamma f$ for any $\gamma \in
\SL_2(\bZ)$. Such orbits $[f]$ may be ordered by their {\em Julia
  invariant}, which is an invariant defined in \cite{julia} for the action
of $\SL_2(\Z)$ on $\Sym_n(\Z^2)$ (see \S \ref{juliadef} for details).
Thus, we also define the family $\fR_J$ to be the multiset of rings
	$$\fR_J = \{R_{[f]} \mid [f] \in \SL_2(\bZ) \backslash
\Sym_n(\bZ^2)\},$$ ordered by Julia invariant $J$, where $J(R_{[f]})
:= J([f])$.
Asymptotics
on the size of $\fR_J$ were obtained by Bhargava--Yang
\cite{BY-Julia}.

In this paper, we compute averages taken over certain families
contained in $\fR_H$ or $\fR_J$. Let $\fR_H^{r_1,r_2} \subset \fR_H$
and $\fR_J^{r_1,r_2} \subset \fR_J$ be the respective subfamilies
consisting of all Gorenstein\footnote{From \cite[Prop.~2.1 and
    Cor.~2.3]{rings} it follows that the ring $R_f$ is Gorenstein if
  and only if $f$ is {\em primitive}, i.e., the coefficients of $f$ do
  not share any common prime factors.} integral domains whose fraction
field has signature $(r_1,r_2)$, i.e., has $r_1$ real embeddings and
$r_2$ pairs of conjugate complex embeddings. Also, let
$\fR^{r_1,r_2}_{H,\max} \subset \fR^{r_1,r_2}_H$
(resp. $\fR^{r_1,r_2}_{J,\max} \subset \fR^{r_1,r_2}_{J}$) be the
subfamily containing all maximal orders. It is worthwhile to note that
a given order $\cO$ in a number field with signature $(r_1,r_2)$ may
occur in $\fR^{r_1,r_2}_H$ or $\fR^{r_1,r_2}_{H,\max}$ an infinite
number of times (up to isomorphism) but only occurs with finite
multiplicity in $\fR^{r_1,r_2}_J$ or $\fR^{r_1,r_2}_{J,\max}$ by a
result of Birch--Merriman \cite{birchmerriman}.

For any subfamilies $\Sigma_H \subseteq \fR_H^{r_1,r_2}$ and $\Sigma_J \subseteq
\fR_J^{r_1,r_2}$, we denote the average number of $2$-torsion elements
of ideal class groups over $\Sigma_H$ ordered by height and over $\Sigma_J$
ordered by Julia invariant as follows:
	\begin{equation} \label{eq:defavg}
	\Avg_H(\Sigma_H,\Cl_2) = \lim_{X\rightarrow \infty} \frac{\displaystyle\sum_{\substack{R_f \in \Sigma_H \\ |H(f)| < X}} |\Cl_2(R_f)|}{\displaystyle\sum_{\substack{R_f \in \Sigma_H \\ |H(f)| < X}} 1} \mbox{\; and \;}\Avg_J(\Sigma_J,\Cl_2) = \lim_{X\rightarrow \infty}  \frac{\displaystyle\sum_{\substack{R_{[f]} \in \Sigma_J \\ |J(f)| < X}} |\Cl_2(R_f)|}{\displaystyle\sum_{\substack{R_{[f]} \in \Sigma_J \\ |J(f)| < X}} 1}
	\end{equation}
where $\Cl_2(R_f)$ denotes the $2$-torsion subgroup of the ideal class group
of $R_f$.  Additionally, we can replace $\Cl_2(R_f)$ with the $2$-torsion subgroup
$\Cl_2^+(R_f)$ of the {\em narrow} class group of $R_f$ in the right hand sides
of the equalities in \eqref{eq:defavg}; we denote these means by $\Avg_H(\Sigma_H,\Cl_2^+)$ and
$\Avg_J(\Sigma_J,\Cl_2^+)$, respectively. The notation
$\Avg_{\ast}(\ast,\ast) \leq c$ will be used to indicate that
the limsups of fractions as in \eqref{eq:defavg}
are bounded by $c$. We then have the following theorem:
\begin{theorem}\label{thm:avgforfields}
Fix an odd integer $n > 3$ and a corresponding signature $(r_1,r_2)$. Then:
\begin{itemize}\label{eqavgstatementth}
	\item[\rm (a)] $\Avg_H(\fR_{H,\max}^{r_1,r_2},\Cl_2) \leq 1+2^{1-r_1-r_2}$ and $\Avg_J(\fR_{J,\max}^{r_1,r_2},\Cl_2) \leq 1+2^{1-r_1-r_2}$, and
	\item[\rm (b)] $\Avg_H(\fR_{H,\max}^{r_1,r_2},\Cl_2^+)\leq 1+{2^{-r_2}}$ and $\Avg_J(\fR_{J,\max}^{r_1,r_2},\Cl_2^+) \leq 1+{2^{-r_2}}$.
\end{itemize}
If the tail estimates in \eqref{eqassumedte} hold, then both $({\rm
  a})$ and $({\rm b})$ are equalities. Additionally, the same upper
bounds $($and conditional equalities$)$ hold when further imposing any finite
set of local conditions on the fields in $\fR_{H,\max}^{r_1,r_2}$ and
$\fR_{J,\max}^{r_1,r_2}$.
\end{theorem}

When $n = 3$, the Julia invariant of a ring $R_f$ associated to a
binary cubic form $f$ coincides with its discriminant, and the family
$\fR_{J}$ is essentially the same as the family of all cubic rings
ordered by discriminant. The mean size of the $2$-torsion subgroup of
class groups of totally real (resp. complex) cubic fields ordered by
discriminant was determined to be $5/4$ (resp. $3/2$) in
\cite{manjulcountquartic}, confirming Conjecture \ref{conj:malle} for
$n = 3$. Additionally, the average number of $2$-torsion elements in the
narrow class groups of totally real cubic fields ordered by
discriminant is $2$, which was proved by Bhargava and the third-named
author \cite{BV3}. On the other hand, even though the family $\fR_{H}$
also contains all cubic rings, each such ring occurs infinitely
often. Nevertheless, we determine that the average number of $2$-torsion
elements in class groups and narrow class groups of cubic fields
ordered by height coincides with the analogous results in
\cite{manjulcountquartic} and \cite{BV3} when ordering by
discriminant.

\begin{theorem}\label{thm:avgforcubicfields} We have 
\begin{itemize}
\item[\rm (a)] $\Avg_H(\fR_{H,\max}^{3,0},\Cl_2) = 5/4$,
\item[\rm (b)] $\Avg_H(\fR_{H,\max}^{1,1},\Cl_2) = 3/2$, and
\item[\rm (c)] $\Avg_H(\fR_{H,\max}^{3,0},\Cl_2^+) = 2$.
\end{itemize}
\end{theorem}

In conjunction with Theorem $1$ of \cite{BV3}, Theorem \ref{thm:avgforcubicfields} gives evidence that
the Cohen-Lenstra-Martinet-Malle heuristics may hold for {any} natural ordering of fields, as they hold when ordering by either discriminant or height. Additionally, Theorem \ref{eqavgstatementth}(b) gives
evidence toward the prediction that the average number of $2$-torsion
elements in the narrow class groups of {\em all} isomorphism classes
of odd degree number fields with fixed signature $(r_1,r_2)$ is equal to
	\begin{equation}1 + 2^{-r_2},
	\end{equation}
	which additionally coincides with heuristics formulated by Dummit--Voight \cite{dummitvoight}.
 
\medskip 

Theorems \ref{thm:avgforfields} and \ref{thm:avgforcubicfields}
immediately imply that most fields within these families have
no nontrivial $2$-torsion elements in their class groups. By applying results of \cite{birchmerriman}, we may quantify
the number of such fields, even while allowing arbitrary splitting conditions at a finite set of primes.

\begin{theorem}\label{cor:oddclassno}
Fix an odd integer $n \geq 3$ and a corresponding signature
$(r_1,r_2)$. Let $S$ be a finite set of primes and for each prime $p
\in S$, fix a degree $n$ \'etale extension $M_p$ of $\bQ_p$.
\begin{itemize}
\item[\rm (a)] There are an infinite number of degree $n$ $S_n$-fields $K$
 with signature $(r_1,r_2)$ such
  that $K \otimes \bQ_p = M_p$ for each $p \in S$, and $K$ has odd
  class number. More precisely,
\begin{equation*}
  \#\bigl\{\ K \ : \ |\Disc(K)| < X \mbox{ and } 2 \nmid|\Cl(K)|\ \bigr\}
\gg X^{\frac{n+1}{2n-2}},
\end{equation*}
where the implied constants depend on $n$ and $S$.
\item[\rm (b)] If $r_2 \geq 1$, then there are an infinite number of
  degree $n$ $S_n$-fields $L$ with
  signature $(r_1,r_2)$ such that $L \otimes \bQ_p = M_p$ for each $p
  \in S$, and $L$ has odd narrow class number. More precisely,
\begin{equation*}
  \#\bigl\{\ L \ : \ |\Disc(L)| < X \mbox{ and } 2 \nmid|\Cl^+(L)| \ \bigr\}
  \gg X^{\frac{n+1}{2n-2}},
\end{equation*}
where the implied constants depend on $n$ and $S$.
\end{itemize}
\end{theorem}

Such results on the infinitude of fields with odd class number
originate with Gauss \cite{gauss}, who proved using genus theory that
the set of quadratic fields with class number indivisible by $2$ are
exactly the quadratic fields with prime discriminant. The first
generalization of Gauss's result to the indivisibility of class
numbers of quadratic fields by odd primes $p$ arise as applications of
the aforementioned results of Davenport--Heilbronn
\cite{davenportheilbronn2}, which imply that at least half of
imaginary quadratic fields and at least $5/6$ of real quadratic fields
have class number indivisible by $3$ when such fields are ordered by discriminant.
Nakagawa--Horie \cite{nakagawahorie} refined the proof of
\cite{davenportheilbronn2} to show that even after imposing certain
congruence conditions at a finite set of primes, the number of
such quadratic fields with class number indivisible by $3$ remains infinite;
this strengthening implies results such as the existence of infinitely
many hyperelliptic curves over $\bQ$ of a given genus with no integral
points. Finally, the results of Bhargava and the third-named author 
\cite{BV-3tors} imply that one can find an infinite number of quadratic fields 
with class number indivisible by $3$ and satisfying {\em any} (nonempty) local specifications 
at a finite set of primes.

In the imaginary quadratic case, Hartung \cite{Hartung} gave another
proof of the infinitude of fields with class number indivisible by $3$
using Kronecker--Weber relations. In conjunction with
trace formula methods, Horie \cite{HorieIwasawa,HorieTrace} extended
these results to determine that for all sufficiently large primes $p$,
there exist infinitely many imaginary quadratic fields with class
number indivisible by $p$ and satisfying prescribed splitting and
ramification conditions at a finite set of (odd) primes. Using the
indivisibility of coefficients of modular forms of half-integer
weight, Bruinier \cite{Bruinier} and Ono--Skinner \cite{onoskinner}
strengthened the result to include most primes $p \geq 5$ and a wider
class of local specifications that could be imposed at a finite number
of primes. Jochnowitz \cite{jochnowitz} also used such methods to
generalize the results of \cite{Hartung,HorieIwasawa,HorieTrace} to
the real quadratic case. The most general result was obtained by Wiles
\cite{wiles-classgroups} and Beckwith \cite{beckwith} using trace
formula methods in conjunction with the geometry of Shimura curves and
the theory of mock modular forms of half-integer weight,
respectively. Applications of such results include unconditional
versions of modularity lifting theorems in the residually reducible
case \cite{skinnerwiles} as well as the nonvanishing of certain
$L$-values associated to elliptic curves with rational torsion points
\cite{vatsal}.

Beyond the case of quadratic fields, the only known result of this
nature is Corollary 3 of \cite{BV3}, which implies that the majority
of cubic fields (of any signature) have odd class number. Theorem
\ref{cor:oddclassno} is the first of its kind to treat infinite (even
multiple) degrees and signatures. Additionally, it immediately implies
the following result concerning the narrow class number, which differs
from the class number at most by a factor of a power of 2.

\begin{corollary}\label{cor:units}
Let $n\geq 3$ be an odd integer. If $r_2 \geq 1$, then there are an
infinite number of degree $n$ $S_n$-fields with
signature $(r_1,r_2)$ for which the narrow class number equals the
class number. In particular, there are an infinite number of such
fields that have units of every signature\footnote{Recall that for any
  number field $K$ with $r_1$ distinct real embeddings, there is a
  signature homomorphism $\cO_K^\times \rightarrow \{\pm 1\}^{r_1}$
  that takes a unit to its {\em signature}, i.e., to the sign of its
  image under each real embedding.}.
\end{corollary}

\medskip

Our methods are not limited to studying class groups of ({maximal} orders in) number
fields; we also study the ideal class groups of general orders in $\fR^{r_1,r_2}_H$ and $\fR^{r_1,r_2}_J$. Specifically, for
each odd $n \geq 3$, we compute on average how many $2$-torsion ideal
classes in the class groups of such orders arise from nontrivial
elements of order $2$ in the {\em ideal groups} of such orders.
More precisely, if $\cO$ is an
order in a number field, let the ideal group $\cI(\cO)$ be the
group of invertible fractional ideals of $\cO$ (which the class
group $\Cl(\cO)$ is a quotient of). Denote the $2$-torsion
subgroups of $\Cl(\cO)$ and $\cI(\cO)$ by $\Cl_2(\cO)$ and
$\cI_2(\cO)$ for any prime $p$. Although $\cI_2(\cO)$ is trivial for
maximal orders $\cO$, this is not always true for non-maximal
orders $\cO$.

In \cite{BV3}, the mean value of the difference
$|\Cl_2(\cO)| - 2^{1-r_1-r_2}|\cI_2(\cO)|$ is determined to be 1,
when averaging over maximal orders $\cO$ in cubic fields of a fixed
signature $(r_1,r_2)$, over all orders in such cubic fields, or even
over certain {acceptable} families of orders defined by local
conditions (in all cases ordered by discriminant).  An analogous result is also known
for $3$-torsion ideal classes of acceptable families of
quadratic orders and fields (see \cite{BV-3tors}). In this paper, we
obtain a similar statement for $\fR^{r_1,r_2}_J$ and $\fR^{r_1,r_2}_H$:

\begin{theorem} \label{thm:avgfororders}
Fix an odd integer $n \geq 3$ and signature $(r_1,r_2)$.
\begin{itemize}
\item[\rm (a)] The average size of
$$|\Cl_2(\cO)| - \frac{1}{2^{r_1+r_2-1}} |\mathcal{I}_2(\cO)|$$
over $\cO \in \fR^{r_1,r_2}_H$ ordered by height or over $\cO \in \fR^{r_1,r_2}_J$ ordered by Julia invariant is $1$.
\item[\rm (b)] The average size of
$$|\Cl_2^+(\cO)| - \frac{1}{2^{r_2}} |\mathcal{I}_2(\cO)|$$
over $\cO \in \fR^{r_1,r_2}_H$ ordered by height or over $\cO \in \fR^{r_1,r_2}_J$ ordered by Julia invariant is $1$.
\end{itemize}
\end{theorem}

In fact, we prove a much stronger statement indicating that the above
averages remain equal to $1$ when taken over any {\em very large}
family in $\fR^{r_1,r_2}_H$ or $\fR^{r_1,r_2}_J$ (see Definition \ref{6.1}). For any {\em acceptable} family in $\fR^{r_1,r_2}_{H}$ or
$\fR^{r_1,r_2}_{J}$ (as defined in \S\ref{acc}), the analogous averages are shown to have an
upper bound equal to 1; furthermore, conditional on the tail estimates in (\ref{eqassumedte}), averages over acceptable families in  $\fR^{r_1,r_2}_{H}$ and
$\fR^{r_1,r_2}_{J}$ also have lower bound equal to 1 (see Theorem \ref{thmacceptablefam}). Some notable acceptable families include  $\fR^{r_1,r_2}_{H,\max}$ and
$\fR^{r_1,r_2}_{J,\max}$ as well as subfamilies of $\fR^{r_1,r_2}_{H,\max}$ and $\fR^{r_1,r_2}_{J,\max}$ that are defined by local conditions at any finite set of primes. 

\medskip

Our strategy for proving Theorems \ref{thm:avgforfields},
\ref{thm:avgforcubicfields}, and \ref{thm:avgfororders} uses Wood's
parametrization \cite{melanie-2nn} of $2$-torsion ideal classes of
rings in $\fR^{r_1,r_2}_H$ and $\fR^{r_1,r_2}_J$ by certain integral
orbits of the representation $\bZ^2 \otimes \Sym^2(\bZ^n)$; we then
determine asymptotic counts of the relevant orbits using
geometry-of-numbers techniques developed by
\cite{manjulcountquartic,manjulcountquintic,5sel}.  However, our
geometry-of-numbers arguments are complicated by the fact that we
simultaneously consider an infinite set of representations, one for
each odd $n \geq 3$, which have increasingly intricate invariant
rings. Similar infinite sets have been handled
previously in \cite{manjul-hyper, bg-hyper, BGW}.
An essential ingredient for our result is a sieve that counts binary
$n$-ic forms that correspond to maximal rings (equivalently, degree $n$ fields). For the family of binary $n$-ic forms
ordered by height, this sieve is carried out in \cite{BSW-part2}, and we carry out an analogous sieve for binary
$n$-ic forms ordered by Julia invariant.

When ordering by height, we study the orbits of
$\SL_n(\bZ)$ acting on the space $\bZ^2 \otimes \Sym_2(\bZ^n)$
of pairs $(A,B)$ of integral $n$-ary quadratic forms. Each such pair gives rise
to an {\em invariant binary $n$-ic form}
	$$f_{(A,B)}(x,y) := \det (Ax-By)$$ when $A$ and $B$ are viewed
as symmetric $n \times n$ matrices. If $R_f \in \fR^{r_1,r_2}_H$ for
some signature $(r_1,r_2)$, then certain {\em projective}
$\SL_n(\bZ)$-orbits of pairs $(A,B)$ with invariant binary $n$-ic
$f_{(A,B)} = f$ are equipped with a composition law coming from the
group structure on the $2$-torsion subgroup of the class group of
$R_f$. (The notion of projectivity is defined in \S2.3.) This
implies that the number of such orbits is determined by the number of
$2$-torsion ideal class elements of $R_f$. Thus, to compute the
averages when ordering by height in Theorem \ref{thm:avgfororders}, we
compare the number of rings (with multiplicity) in $\fR^{r_1,r_2}_H$
of bounded height to the number of relevant $\SL_n(\bZ)$-orbits whose
binary $n$-ic invariant is bounded by the same height. To obtain
Theorems \ref{thm:avgforfields} and \ref{thm:avgforcubicfields}, we
restrict to maximal orders, namely those rings $R_f \in
\fR^{r_1,r_2}_{H,\max}$; however, a conjectural tail estimate is
required to obtain a lower bound.

When ordering by Julia invariant, we count the number of $\SL_2(\bZ) \times \SL_n(\bZ)$-orbits of $\bZ^2 \otimes \Sym_2(\bZ^n)$ relative to the number of $\SL_2(\bZ)$-orbits of $\Sym_n(\bZ^2)$. As described above, the rings $R_f$ associated to a binary $n$-ic form $f$ are invariant under the action of $\SL_2(\bZ)$ on $f$, i.e., for any $\gamma f \in [f] = \SL_2(\bZ)\cdot f$, we have  $R_{\gamma f} \cong R_f$. It follows from \cite{melanie-2nn} that if $\cO_{[f]} \in \fR^{r_1,r_2}_J$ for some signature $(r_1,r_2)$, then {\em projective} $\SL_2(\bZ)\times \SL_n(\bZ)$-orbits of pairs of $n$-ary quadratic forms $(A,B)$ with $[f_{(A,B)}] = [f]$ are in bijection with $2$-torsion elements of the class group of $\cO_{[f]}$. We then use the same geometry-of-numbers methods utilized when ordering by height to conclude Theorems \ref{thm:avgforfields} and \ref{thm:avgfororders} when ordering by Julia invariant. 
Note that when $n = 3$, the Julia invariant coincides with the discriminant of a binary cubic form, and so our argument can be viewed as a generalization of that given in \cite{BV3}.

We now give a short description of the organization of the
paper. In Section \ref{sec:param}, we recall and expand on the details
of the construction of rings $R_f$ of rank $n$ from binary $n$-ic forms
$f$ given in \cite{nakagawa,rings}. We also describe the correspondence given in \cite{melanie-2nn} between
$\SL_n$-orbits of pairs of $n$-ary quadratic forms and order $2$ ideal
classes of such rings $R_f$. Section 3 discusses asymptotic counts of acceptable families in $\fR^{r_1,r_2}_H$ and $\fR^{r_1,r_2}_J$. Section
\ref{sec:counting} focuses on using geometry-of-numbers methods to
count the projective integral orbits of pairs of $n$-ary
quadratic forms whose binary
$n$-ic invariant $f$ is contained in $\fR^{r_1,r_2}_H$ or $\fR^{r_1,r_2}_J$. In
Section \ref{sec:sieve}, we describe several sieves that allow us to
restrict our count from Section \ref{sec:counting} to orbits that
correspond to invertible ideal classes in orders (or maximal
orders). Finally, in Section \ref{sec:proofmain}, the analytic methods
in Sections \ref{sec:counting} and \ref{sec:sieve} are combined
with the algebraic interpretation of the orbits given in Section
\ref{sec:param} to conclude the main results.

\subsubsection*{Acknowledgments}
We thank Manjul Bhargava, Christopher Delaunay, Robert Harron, Gunter Malle, Michael
Stoll, Xiaoheng Wang, Melanie Matchett Wood, and Myungjun Yu for helpful
conversations and comments. We also thank the anonymous referees for
many useful suggestions. The first-named author was supported by NSF grant
DMS-1406066 and the third-named author by NSF grant DMS-1502834.

\section{Parametrizations of $2$-torsion ideal classes and composition laws} \label{sec:param}

Let $n\geq 3$ be a fixed odd integer.  In this section, we begin by
recalling from \cite{nakagawa,rings} how rings of rank $n$ naturally
arise from integral binary $n$-ic forms. We then recall the
parametrization given in \cite{melanie-2nn} of $2$-torsion ideal
classes in such rings by orbits of pairs of $n$-ary quadratic forms.
In \S \ref{sec:composition}, we describe a composition law for certain
orbits of pairs of $n$-ary quadratic forms, arising from the group law
on ideal classes in rings. In \S \ref{sec:reducible}, we discuss
``reducible'' elements in the space of such integral pairs and the
properties of the corresponding $2$-torsion ideal classes via
the parametrization; these are elements that will be excluded in the
volume computations in later sections. Finally, in \S
\ref{sec:orbitstab}, we use a rigidified version of the
parametrization theorem in \cite{melanie-2nn} over principal ideal
domains to explicitly describe the stabilizers and orbits of these
representations for a few specific base rings.

\subsection{Rings associated to binary $n$-ic forms}\label{rings}

We first describe the construction of a rank $n$ ring over $\Z$ and
ideals from an integral binary $n$-ic form. Let $f(x,y) = f_0 x^n +f_1
x^{n-1}y +\cdots + f_ny^n$, where $f_i \in \bZ$. We begin with the case where
$f_0 \neq 0$, and let $B_{f_0} = \bZ[\frac{1}{f_0}]$. Define the ring $R_f$ as a
subring of $B_{f_0}[\theta]/f(\theta,1)$, generated as a $\bZ$-module
as
	\begin{equation} \label{eq:basisRf}
		R_f = \langle 1, f_0 \theta, f_0 \theta^2 + f_1
                \theta, \dots, f_0 \theta^{n-1} + f_1 \theta^{n-2} +
                \dots + f_{n-2} \theta \rangle.\end{equation} For $k
        > 0$, define $\zeta_k = f_0 \theta^k + \cdots + f_{k-1}
        \theta$, and let $\zeta_0 = 1$. It is shown in both \cite{nakagawa} and
        \cite{rings} that $R_f = \langle \zeta_0, \ldots,
        \zeta_{n-1}\rangle$ is closed under multiplication and thus is a ring.  We
        define the following $\bZ$-submodule of
        $B_{f_0}[\theta]/f(\theta,1)$:
	\begin{equation} \label{eq:If}
		\If = \langle 1, \theta, \zeta_{2},\dots, \zeta_{n-1}\rangle.
	\end{equation}
As shown in \cite{nakagawa,rings}, the module $\If$ is closed under
multiplication by elements of $R_f$ and thus is an ideal of $R_f$. It
is easy to check that for $0 \leq k \leq n-1$, we have
	\begin{equation}\label{eq:Ifk}
	\If^{k} = \langle 1, \theta, \theta^2, \hdots, \theta^{k},
        \zeta_{k+1},\cdots, \zeta_{n-1}\rangle
	\end{equation}
as a $\bZ$-submodule of $B_{f_0}[\theta]/f(\theta,1)$. For $n$ odd, the ideal $\If^{n-3}$ is a square of the ideal $\If^{\frac{n-3}{2}}$, which has the following explicit basis as a $\bZ$-module:
	$$\If^{\frac{n-3}{2}} = \langle 1, \theta, \theta^2, \hdots,
\theta^{\frac{n-3}{2}}, \zeta_{\frac{n-3}{2}+1}, \hdots, \zeta_{n-1}
\rangle.$$ Additionally, there is a natural action of $\gamma \in
\GL_2(\bZ)$ on the set of binary $n$-ic forms $f$ sending $\gamma
\cdot f(x,y) = f((x,y)\gamma)$; under this action, the ring $R_f$ and
the ideal $\If$ (and its powers) are invariant (up to isomorphism). If
$f$ is irreducible, then $R_f$ is an order of
$\bQ[\theta]/f(\theta,1)$, and the discriminants of $R_f$ and $f$
coincide \cite[Proposition 1.1]{nakagawa}. In addition, the form $f$
is {\em primitive} (i.e., the gcd of its coefficients is 1) if and
only if $R_f$ is Gorenstein, which is equivalent to the property that
$\If$ is an invertible fractional ideal \cite[Prop.~2.1 and
  Cor.~2.3]{rings}.

In fact, by recording the basis \eqref{eq:Ifk}, the ideals $\If^k$ may
be considered as {\em based ideals} of $R_f$, i.e., ideals of $R_f$
along with an ordered basis as a rank $n$ $\Z$-module.  The {\em norm}
$\Norm(\II)$ of a based ideal $\II$ of $R_f$ is the determinant of the
$\bZ$-linear transformation taking the chosen basis of $\II$ to the
basis of $R_f$ given by \eqref{eq:basisRf}.

We also introduce dual elements to $\theta^k$ for all $0 \leq k \leq
n-1$. Let $\{
\check{\theta}_0,\check{\theta}_1,\hdots,\check{\theta}_{n-1} \}$
be the $B_{f_0}$-module basis of
$\Hom_{B_{f_0}}(B_{f_0}[\theta]/f(\theta,1),B_{f_0})$ dual to $\{
1, \theta, \theta^2, \hdots, \theta^{n-1}\}$. Additionally,
define $\check{\zeta}_{n-1} := \frac{\check{\theta}_{n-1}}{f_0}$, and
note that $\check{\zeta}_{n-1}(\zeta_k) = \delta_{k,n-1}$ for all $0
\leq k \leq n-1$.  In \cite[Proposition 2.1]{melanie-2nn}, Wood
computes that for any $r \in B_{f_0}[\theta]/f(\theta,1)$ and
$0 \leq k \leq n - 2$,
	\begin{equation}\label{melanie}
	\check{\theta}_{k}(r) = \check{\zeta}_{n-1}(\zeta_{n-1-k} r) +
        f_{n-1-k} \check{\zeta}_{n-1}(r),
        \end{equation}
which will be useful for computations in the following section.

\begin{remark} \label{rmk:f0zero}
If $f_0 = 0$ but $f \not\equiv 0$, there exists a $\GL_2(\Z)$-transformation
that takes $f$ to another binary $n$-ic form $f'$ with a nonzero leading coefficient. To
obtain the ring $R_f$ and the ideal class $\If$ (which are, up to isomorphism, $\GL_2(\Z)$-invariant), 
one may use the above constructions for $f'$ (see \cite[\S 2]{rings}).
\end{remark}

The above construction holds if one replaces $\bZ$ with any integral
domain $T$ (see \cite{rings}); this gives an explicit way of associating a
ring $R_f$, which is rank $n$ as a $T$-module, and a distinguished
(based) ideal $\If$ of $R_f$ to a binary $n$-ic form over $T$. We
refer to $R_f$ as the {\em ring associated to $f$} and $\If$ as the
{\em distinguished ideal} of $R_f$ or $f$. Geometrically, for nonzero
forms $f$, the ring $R_f$ is the ring of functions on the subscheme
$X_f$ of $\PP^1_T$ cut out by the binary $n$-ic form $f$, and the
ideal $\If^k$ is the pullback of $\mathcal{O}(k)$ from $\PP^1_T$ to
$X_f$ (see \cite[Theorem 2.4]{rings}).

We are interested in counting the $2$-torsion ideal classes of the rings
$R_f$ associated to irreducible forms $f$ when $n$ is odd. A key
ingredient is a parametrization of such ideal classes in terms of
pairs of $n \times n$ symmetric matrices, which we recall next.

\subsection{Parametrization of order 2 ideal classes in $R_f$} \label{par}

For any base ring $T$, let $U(T) = \Sym_n (T^2)$ denote the space of binary $n$-ic forms with coefficients in $T$. Let
$V(T) = T^2 \otimes \Sym_2 (T^n)$ denote the space of pairs
$(A,B)$ of symmetric $n \times n$ matrices with coefficients $a_{ij}$
of $A$ and $b_{ij}$ of $B$ in $T$ (for $1 \leq i,j \leq n$) where
$a_{ij} = a_{ji}$ and $b_{ij} = b_{ji}$. The group $\SL_n(T)$ acts
naturally on $V(T)$, where $\gamma \in \SL_n(T)$ acts on $(A,B)$ by
	\begin{equation} \label{eq:GLnaction}
	\gamma (A,B) = (\gamma A \gamma^t, \gamma B \gamma^t).
	\end{equation}
The map $\pi: V(T) \rightarrow U(T)$ sending $(A,B) \mapsto
\det(Ax - By)$ is clearly $\SL_n(T)$-equivariant. We call $f_{(A,B)} := \pi(A,B)$ the {\em binary $n$-ic invariant} 
or {\em resolvent form} of the pair $(A,B)$
(or of the $\SL_n(T)$-equivalence class of $(A,B)$). Recall that a binary $n$-ic
form $f$ is {\em nondegenerate} if and only if its discriminant $\Delta(f)$ is nonzero,
and we will call the pair $(A,B)$ {\em nondegenerate} if and only if $f_{(A,B)}$ is.
In \cite[Thm.~1.3]{melanie-2nn}, Wood describes the $\SL_n(\Z)$-orbits of $V(\Z)$ in terms
of fractional ideals of the rings $R_f$ from \S \ref{rings}:

\begin{thm}[\cite{melanie-2nn}] \label{thm:paramZ}
Let $f \in U(\Z)$ be a nondegenerate primitive binary $n$-ic form with integral coefficients. Then there is a bijection between
$\SL_n(\Z)$-orbits of $(A,B) \in V(\Z)$ with $f_{(A,B)} = f$ and equivalence classes of pairs $(\II,\delta)$ where $\II$ is a fractional
ideal of $R_f$ and $\delta \in (R_f \otimes_\Z \Q)^\times$ with $\II^2 \subset \delta \If^{n-3}$ as ideals and $\Norm(\II)^2 = \Norm(\delta) \Norm(\If^{n-3})$.
Two pairs $(\II,\delta)$ and $(\II',\delta')$ are equivalent if there exists $\kappa \in (R_f \otimes_\Z \Q)^\times$ such that $\II' = \kappa \II$ and $\delta' = \kappa^2 \delta$.
\end{thm}

For forms $f$ with $f_0 \neq 0$ (see Remark \ref{rmk:f0zero}), we now explicitly describe the bijective
map of Theorem \ref{thm:paramZ}, as some of these computations will be
needed in \S \ref{sec:reducible}. Fix a primitive nondegenerate binary cubic form $f(x,y) = f_0 x^n +f_1 x^{n-1}y
+ \dots + f_ny^n\in U(\bZ)$ with $f_0 \neq 0$, and let $R_f$ denote the ring described in \eqref{eq:basisRf}.

We begin by constructing an element of $V(\bZ)$ from a pair
$(\II,\delta)$ where $\II$ denotes a fractional ideal of $R_f$ and $\delta$ denotes an invertible
element of $R_f \otimes_{\bZ} \bQ$ such that $\II^2 \subset \delta
\If^{n-3}$ and $\Norm(\II)^2 = \Norm(\delta) \Norm(\If^{n-3})$. Under these assumptions, we can define a map
	\begin{align}\label{abstractvarphi}
	\varphi: \II \otimes_{R_f} \II &\rightarrow \If^{n-3} \\
	\alpha \otimes \alpha' &\mapsto \frac{\alpha \alpha'}{\delta}.\nonumber
    \end{align}
For the $\bZ$-module $\langle 1,\theta,\dots,\theta^{n-3}\rangle$, there is a
quotient map $\If^{n-3} \rightarrow \If^{n-3}/\langle
1,\theta,\dots,\theta^{n-3}\rangle$, and when $\varphi$ is composed with this quotient map, it gives a symmetric bilinear map that corresponds to an $\SL_n(\Z)$-orbit of $
V(\bZ)$. Equivalently, let $\alpha_1, \dots, \alpha_n$  in  $R_f \otimes_{\bZ} \bQ$ denote elements that generate $\II$ over $\bZ$ and for which the change-of-basis
matrix from $\langle \zeta_0, \zeta_1, \dots, \zeta_{n-1} \rangle$ to 
$\langle \alpha_1, \dots, \alpha_n \rangle$ has positive determinant. 
From the assumption that $\II^2 \subset \delta \If^{n-3}$,
we have that for all $i,j \in \{1,\ldots, n\}$,
	\begin{equation}\label{explicitvarphi}
	\frac{\alpha_i\alpha_j}{\delta} = c_{ij}^{(0)} +
        c_{ij}^{(1)} \theta + \hdots + c_{ij}^{(n-3)} \theta^{n-3} +
        b_{ij} \zeta_{n-2} + a_{ij} \zeta_{n-1}
	\end{equation}
where $a_{ij},b_{ij},c_{ij}^{(k)} \in \bZ$ for $0 \leq k \leq
n-3$. Then $(A,B) = ((a_{ij}),(b_{ij}))$ yields the desired pair of
integral symmetric $n \times n$ matrices.

\smallskip

To describe the reverse map, let $(A,B) \in V(\bZ)$ satisfy $\pi(A,B) = f$, and denote the 
coefficients of $A$ as $a_{ij}$ and of $B$ as $b_{ij}$. Note that $\det A
= f_0$, so  requiring $f_0 \neq 0$ is equivalent to
requiring $A$ to be invertible. We want to construct a fractional
ideal $\II$ of $R_f$ along with an element $\delta \in (R_f
\otimes_{\bZ} \bQ)^\times$ such that $I^2 \subset \delta I_f^{n-3}$ and $\Norm(I)^2 = \Norm(\delta)\Norm(I_f^{n-3})$. Theorem 5.7 of \cite{melanie-2nn} implies that it is equivalent to give a $\bZ$-basis $\langle \alpha_1, \dots, \alpha_n \rangle$ for $\II$ and a map of $R_f$-modules $\varphi: I \otimes_{R_f} I \rightarrow I_f^{n-3}$ such that the composition
	\begin{equation}\label{composition}I \otimes_{\bZ} I \rightarrow I \otimes_{R_f} I \rightarrow I_f^{n-3} \rightarrow \If^{n-3}/\langle 1,\theta,\dots,\theta^{n-3}\rangle\end{equation}
is equal to $(A,B)$ when written in terms of $\langle \alpha_1, \dots, \alpha_n \rangle$. Indeed, independent of the choice of $i$ and $j$ in $\{1, \dots, n\}$, we have the equality $\delta = \frac{\alpha_i\alpha_j}{\varphi(\alpha_i\otimes\alpha_j)}$. (This is due to the fact that any map $I \otimes_{R_f} I \rightarrow I_f^{n-3}$ factors through an injective map $I^2 \rightarrow I_f^{n-3}$, which must be multiplication by an invertible element of $R_f \otimes_{\bZ} \bQ$.) Thus, we would like to describe $I$ in terms of the $\bZ$-basis $\langle \alpha_1, \dots, \alpha_n\rangle$ and construct the map $\varphi$. 

If the composition of maps in \eqref{composition} corresponds to $(A,B)$ relative to a $\bZ$-basis $\langle \alpha_1, \dots,\alpha_n\rangle$, then the map $I\otimes_{\bZ} I \rightarrow I \otimes_{R_f} I \rightarrow I_f^{n-3}$ can be described on elements of the $\bZ$-basis $\langle \alpha_i \otimes \alpha_j \rangle$ of $I\otimes I$ as
\begin{equation}\begin{array}{rcl} \label{eqalphabydelta}
  \displaystyle\varphi(\alpha_i \otimes\alpha_j) & = &
  \displaystyle\sum_{k=0}^{n-3}c_{ij}^{(k)}\theta^k+b_{ij} \zeta_{n-2} + a_{ij} \zeta_{n-1}
\\[.1in] 
&=&\displaystyle
c_{ij}^{(0)}+\sum_{k=1}^{n-3}(c_{ij}^{(k)}+b_{ij}f_{n-k-2}+a_{ij}f_{n-k-1})\theta^k \\
& & \; + \; (b_{ij}f_{0}+a_{ij}f_{1})\theta^{n-2}+a_{ij}f_0\theta^{n-1},
\end{array}
\end{equation}
where $c_{ij}^{(k)}$ are integers for $k \in \{0,\dots,n-3\}$ and $1 \leq i,j\leq n$. Thus, the coefficients $c_{ij}^{(k)}$ must satisfy
	$$c_{ij}^{(k)} = \begin{cases} \check{\theta}_k\left(\varphi(\alpha_i\otimes\alpha_j)\right) - f_{n-k-2} \cdot b_{ij} - f_{n-k-1} \cdot
a_{ij} & \mbox{ if $1 \leq k \leq n-3$} \\
\check{\theta}_k\left(\varphi(\alpha_i\otimes\alpha_j)\right) & \mbox{ if $k = 0$}\end{cases}.$$
Using equation (\ref{melanie}), we then have that $c_{ij}^{(k)}$ for $k > 0$ must satisfy
	\begin{eqnarray*}
	c_{ij}^{(k)} &=& f_{n-k-1}
        \check{\zeta}_{n-1}\left(\varphi(\alpha_i \otimes
          \alpha_j)\right) +
        \check{\zeta}_{n-1}\left(\zeta_{n-k-1} \cdot \varphi(\alpha_i \otimes
          \alpha_j)\right) - f_{n-k-2} \cdot b_{ij} -
        f_{n-k-1} \cdot a_{ij} \\ &=&
        \check{\zeta}_{n-1}\left(\zeta_{n-k-1} \cdot \varphi(\alpha_i \otimes
          \alpha_j)\right) - f_{n-k-2} \cdot
        b_{ij}\\ &=&\check{\zeta}_{n-1}\left((f_0
          \theta^{n-k-1} + f_1 \theta^{n-k-2} + \hdots +f_{n-k-2}
          \theta)\cdot\varphi(\alpha_i\otimes\alpha_j)\right) -
        f_{n-k-2} \cdot b_{ij}. \end{eqnarray*}
        The middle equality
        follows from the fact that
        $\check{\zeta}_{n-1}\left(\varphi(\alpha_i\otimes\alpha_j)\right) = a_{ij}$ by equation
        (\ref{eqalphabydelta}).
By \cite[Proposition 3.3]{melanie-2nn}, if we write an element $\alpha$ of $\II$ as a
row vector $(a_1,a_2,\dots,a_n)$ relative to the $\bZ$-basis $\langle \alpha_1, \dots, \alpha_n
\rangle$ corresponding to $(A,B)$, then $\theta \in B_{f_0}[\theta]/f(\theta,1)$ must act on $\II$ by right multiplication by
$BA^{-1}$, i.e.,
	$$\theta \cdot \alpha =  (a_1, a_2, \dots, a_n ) \cdot BA^{-1}.$$ 
Thus, if we create $n-2$ matrices $C^{(k)}$ such that its $(i,j)^{\text{th}}$ entry is equal to $c_{ij}^{(k)}$, then we have for $k > 0$,
\begin{eqnarray}\label{eqCk}
	C^{(k)} &=& (f_0\cdot(BA^{-1})^{n-k-1} + f_1 \cdot
        (BA^{-1})^{n-k-2} + \hdots + f_{n-k-2} \cdot BA^{-1})A - 
        f_{n-k-2}B \nonumber \\
        &=&(f_0 \cdot (BA^{-1})^{n-k-2} + f_1 \cdot
(BA^{-1})^{n-k-3} + \hdots + f_{n-k-3} \cdot BA^{-1})B.
\end{eqnarray}
Additionally,
\begin{eqnarray*}\label{eqC0}
	C^{(0)} &=& (f_0\cdot(BA^{-1})^{n-1} + f_1 \cdot
        (BA^{-1})^{n-2} + \hdots + f_{n-2} \cdot BA^{-1}  
       + f_{n-1})A.
       \end{eqnarray*}	
Furthermore, since the action of $\theta$ gives the action of $R_f$ on $\II$, this completely determines the map $\varphi$ and $I$ as an $R_f$-module. By Propositions 5.1 and 5.4 of \cite{melanie-2nn}, this implies that $\II$ can be realized as a fractional ideal, and thus there is a well-defined element of $(R_f \otimes_{\bZ} \bQ)^\times$ satisfying
	$$\delta = \frac{\alpha_i \alpha_j}{\varphi(\alpha_i \otimes \alpha_j)},$$ independent of the choice of $i$ and $j$. Additionally, for each $1
\leq j \leq n$, we have that the $\alpha_i$ satisfy the following ratios:
	\begin{align*}
	\alpha_1 : \alpha_2 : \hdots : \alpha_{n-1} : \alpha_n 
	= & \ c_{1,j}^{(0)} + \dots + c_{1,j}^{(n-3)} \theta^{n-3} 
	+ b_{1,j} \zeta_{n-2} + a_{1,j} \zeta_{n-1} \ : \\
	&\ c_{2,j}^{(0)} + \dots +
        c_{2,j}^{(n-3)} \theta^{n-3} + b_{2,j} \zeta_{n-2} + a_{2,j}
        \zeta_{n-1} : \hdots : \\
    &\ c_{n-1,j}^{(0)} + \dots + 
        c_{n-1,j}^{(n-3)} \theta^{n-3} + b_{n-1,j} \zeta_{n-2} +
        a_{n-1,j} \zeta_{n-1} \ : \\
    &\ c_{n,j}^{(0)} + \dots + 
        c_{n,j}^{(n-3)} \theta^{n-3} + b_{n,j} \zeta_{n-2} + a_{n,j}
        \zeta_{n-1}.
	\end{align*}
The ratios must be independent of the choice of $j$, so this in conjunction with $\delta$ determines
$\langle \alpha_1, \alpha_2, \hdots, \alpha_n \rangle$.
The action of $\SL_n(\bZ)$ on $V(\bZ)$ corresponds to the action $g_n
\in \SL_n(\bZ)$ on the chosen basis for $\II$ which sends
	\begin{equation}\label{SLaction}
	\langle \alpha_1, \alpha_2, \hdots, \alpha_n \rangle \mapsto
        \langle \alpha_1, \alpha_2,\hdots,\alpha_n\rangle \cdot g_n^t.
	\end{equation}
Thus, the ideal $\II$ is invariant under the action of $\SL_n(\bZ)$.

\subsection{Composition of elements of $V(\bZ)$ with the same binary $n$-ic invariant} \label{sec:composition}

Let $\cO$ be an $S_n$-order, i.e. an order in a degree $n$ $S_n$-number field $K$ over $\Q$. Consider the
set of pairs $(\II,\delta)$, where $\II$ is a fractional ideal of $\cO$,
$\delta\in K^\times$, $\II^2\subset (\delta)$, and
$N(\II)^2=N(\delta)$. Recall that we called two such pairs $(\II,\delta)$
and $(\II',\delta')$ {\em equivalent} if there exists $\kappa\in
K^\times$ such that $\II'=\kappa \II$ and $\delta'=\kappa^2\delta$. We
have a natural law of composition on equivalence classes of such pairs
given by
\begin{equation}\label{eqcomplaw}
(\II,\delta)\circ (\II',\delta')=(\II\II',\delta\delta').
\end{equation}
We say that a pair $(\II,\delta)$ is {\em projective} if $\II$ is
projective as an $\cO$-module, i.e., if $\II$ is invertible as a
fractional ideal of $\cO$; the pair $(\II,\delta)$ is projective if and
only if $\II^2=(\delta)$. The set of equivalence classes of projective pairs $(\II,\delta)$ for
$\cO$ forms a group under the composition law \eqref{eqcomplaw}, which
we denote by $\rmH(\cO)$.

There exists a natural group homomorphism from $\rmH(\cO)$ to
$\Cl_2(\cO)$, given by sending the pair $(\II,\delta)$ to the ideal
class of $\II$. This map is clearly well defined and surjective. The
kernel consists of equivalence classes of pairs $(\II,\delta)$ where $\II$
is a principal ideal; each such equivalence class has
a representative of the form $(\cO,\delta)$ where $\delta$ is a
norm $1$ unit. Therefore, we obtain the exact sequence
\begin{equation}\label{eqexactseqcl}
1\to\frac{\cO^\times_{N=1}}{({\cO^\times})^2}\to \rm\rmH(\cO) \to \Cl_2(\cO) \to 1,
\end{equation}
which implies that $\rmH(\cO)$ is an extension of the $2$-torsion subgroup of the class
group of $\cO$. Using Dirichlet's unit theorem and the fact that
$-1\in \cO^\times$ has norm $-1$, we immediately obtain the following lemma:
\begin{lemma}\label{lemsizeH}
Let $\cO$ be an order in an $S_n$-number field of degree $n$ and
signature $(r_1,r_2)$. Then $|\rmH(\cO)|=2^{r_1+r_2-1}|\Cl_2(\cO)|$.
\end{lemma}

We next compare certain elements of $\rmH(\cO)$ to the $2$-torsion subgroup $\Cl_2^+(\cO)$
of the {\em narrow} class group $\Cl^+(\cO)$ of $\cO$. Recall that $\Cl^+(\cO)$ is 
the quotient of the ideal group $\cI(\cO)$ of $\cO$ by the group
$P^+(\cO)$ of {\em totally positive} principal fractional ideals of $\cO$,
i.e., ideals of the form $a\cO$ where $a$ is an element of
$\Frac(\cO)^\times$ such that $\sigma(a)$ is positive for every
embedding $\sigma: \Frac(\cO) \rightarrow \bR$. We say that such an
element $a$ is {\em totally positive} and denote this condition by $a \gg 0$.

\begin{lemma} \label{lem:sizeHplus}
Let $\cO$ be an order in a degree $n$ $S_n$-number field with signature $(r_1,r_2)$. If $\rmH^+(\cO)$ denotes
the subgroup of $\rmH(\cO)$ consisting of projective pairs
$(\II,\delta)$ such that $\delta \gg 0$, then
	\begin{equation} \label{eq:HCl}
	|\rmH^+(\cO)| = 2^{r_2} |\Cl_2^+(\cO)|.
	\end{equation}
\end{lemma}

\begin{proof}
Let $\cO^{\times}_{\gg0}$ denote the totally positive units of $\cO$, and define $\sgn: \cO^\times \rightarrow \{\pm 1\}^{r_1}$ as the signature homomorphism, which takes a unit to the sign of its image under each real embedding $\sigma: \Frac{\cO} \rightarrow \bR$. Let $r$ be the nonnegative integer satisfying $|\mathrm{Image}(\sgn)| = 2^r$ and let
	$$\cC_2^{\gg0}(\cO) = \{[\II] : \textrm{there exists } \delta \gg 0 \textrm{ such that } \II^2 = (\delta)\}$$ 
be the set of equivalence class of ideals whose square is totally positive, where two ideals are equivalent if they differ by a principal ideal (in the usual sense).
We then have the following commutative diagram of exact sequences:
\[
\begin{tikzcd}[column sep=1.5pc, row sep = 2.5pc]
1 \arrow{rd}{} &  & &&  & & 1 \arrow{ld}{}  \\
1 \arrow{r}{} & \cO^\times_{\gg0}/(\cO^\times)^2 \arrow{r}{} \arrow{rd}{} &  \cO^\times/(\cO^\times)^2 \arrow{rr}{\sgn} & & \{\pm 1\}^{r_1}  \arrow{r}{} & \{\pm 1\}^{r_1}/\sgn(\cO^\times) \arrow{r}{} \arrow{ld}{} &1 \\
& & \rmH^+(\cO) \arrow{rd}[swap]{\alpha} & & \Cl_2^+(\cO) \arrow{ld}{\beta} & & \\ 
& & & \cC_2^{\gg0}(\cO) \arrow{ld}{} \arrow{rd}{} &  & & \\
&  & 1 & & 1 & &
\end{tikzcd}
\]
where the map $\alpha$ sends a pair $(\II,\delta)$ with $\delta \gg 0$ to the equivalence class $[\II]$, and the map $\beta$ sends a coset $\II + P^+(\cO)$ to the equivalence class $[\II]$.
We have that $|\cO^\times/(\cO^\times)^2| = 2^{r_1+r_2}$ and $|\{\pm 1\}^{r_1}/\sgn(\cO^\times)| = 2^{r_1 - r}$, so $|\cO^\times_{\gg0}/(\cO^\times)^2| =2^{r_1-r+r_2}$. The equality \eqref{eq:HCl} follows immediately.
\end{proof}

We now relate {\em projective} orbits of $V(\Z)$ to the size of the
$2$-torsion subgroup of the ideal class group of the corresponding
rings.
We say that
a pair $(A,B) \in V(\Z) \cap \pi^{-1}(f)$ is {\em projective} if the
corresponding pair $(\II,\delta)$ under the bijection of
Theorem~\ref{thm:paramZ} is projective. We then have the following
result:

\begin{prop}\label{prophr}
Let $\cO$ be an $S_n$-order corresponding to an integral,
nondegenerate, irreducible, and primitive binary $n$-ic form $f$. Then
$\rmH(\cO)$ is in natural bijection with the set of projective
$\SL_n(\Z)$-orbits on $V(\Z)\cap\pi^{-1}(f)$.  The number of such
projective orbits is equal to
$$2^{r_1+r_2-1}\ |\Cl_2(\cO)|,$$ where $(r_1,r_2)$ is the signature of the
fraction field of $\cO$.
\end{prop}
\begin{proof}
From Theorem \ref{thm:paramZ}, projective orbits in $V(\Z)\cap\pi^{-1}(f)$
are in bijection with pairs $(\II,\delta)$, where $\II$ is a fractional
ideal of $\cO$, $\delta\in K^\times$ and $\II^2=\delta \If^{n-3}$. The set
of such pairs is clearly in bijection with $\rmH(\cO)$ by simply sending
$(\II,\delta)$ to $\Bigl( \II\cdot\If^{-\frac{n-3}{2}},\delta \Bigr)$. The second assertion
of the proposition now follows immediately from Lemma \ref{lemsizeH}.
\end{proof}

\subsection{Reducible elements in $V(\bZ)$} \label{sec:reducible}

We say that an element $(A,B)\in V(\Q)$ is {\it reducible} if the
quadrics in $\P^{n-1}(\Q)$ corresponding to $A$ and $B$ have a common rational
isotropic subspace of
dimension $(n-1)/2$ in $\P^{n-1}(\Q)$. The condition of reducibility
has the following arithmetic significance:

\begin{thm}\label{threducible}
Let $(A,B)$ be a projective element of $V(\bZ)$ whose binary $n$-ic
invariant is primitive, irreducible, and nondegenerate, and let
$(\II,\delta)$ denote the corresponding pair as given by Theorem
\ref{thm:paramZ}. Then $(A,B)$ is reducible if and only if $\delta$ is a
square in $(R_f \otimes_\Z \bQ)^\times$.
\end{thm}
\begin{proof}
Suppose first that $\delta=r^2$ is the square of an invertible element
in $(R_f\otimes\Q)^\times$. By replacing $I$ with $r^{-1}I$ and
$\delta$ with $r^{-2}\delta$, we may assume that $\delta=1$. Let
$\alpha_1,\ldots\alpha_{\frac{n-1}2}$ be a $\Z$-basis for $I \cap
(\Z\oplus\Z\theta\oplus\cdots\oplus\Z\theta^{\frac{n-3}2})$, and extend
it to a basis $\alpha_1,\ldots,\alpha_n$ of $I$. It follows from
\eqref{explicitvarphi} that, with these coordinates, we have
$a_{ij}=b_{ij}=0$ for $1\leq i,j \leq (n-1)/2$, which is sufficient
for $(A,B)$ to be reducible.

Now assume that $(A,B)$ is reducible; we would like to prove that $\delta$ is a
square. Let $x_1,x_2, \ldots, x_n$ denote a set of coordinates for
$\P^{n-1}$. By replacing $(A,B)$ with an $\SL_n(\Q)$-translate if
necessary, we may assume that the common isotropic subspace is the one
generated by $x_1,\ldots,x_{(n-1)/2}$. This implies that
$a_{ij}=b_{ij}=0$ for $1\leq i,j\leq (n-1)/2$.
From
\eqref{eqalphabydelta} and \eqref{eqCk}, we see that the quantity $\alpha_i\alpha_j/\delta$ is given by
the $ij$th coordinate of the matrix
\begin{eqnarray}
D &:=& C^{(0)} + \displaystyle\Bigl(\sum_{k=1}^{n-3}(C^{(k)}+f_{n-k-2}B+f_{n-k-1}A)\cdot\theta^k\Bigr) + (f_0B + f_1 A)\cdot \theta^{n-2} + f_0 A \cdot \theta^{n-1} \nonumber\\
&=&\displaystyle\sum_{k=0}^{n-1}\Bigl(\sum_{j=0}^{n-k-1}f_{n-k-j-1}(BA^{-1})^j\Bigr)
A\cdot\theta^k \nonumber \\
&=&\displaystyle\sum_{j,k\geq 0}^{j+k\leq n-1}f_{n-j-k-1}(BA^{-1})^jA\cdot\theta^k \label{eq:Dformula}
\end{eqnarray}
where $f = f_0 x^n + f_1 x^{n-1} y + \cdots + f_n y^n$ is the binary $n$-ic invariant of $(A,B)$.
(Note that $A$ is invertible because $f$ is assumed to be irreducible, so $f_0 = \det A \neq 0$.)

We now prove that the $11$-coefficient $d_{11}$ of $D$ is a square using the fact that $a_{ij}=b_{ij}=0$ for $1\leq i,j\leq (n-1)/2$.
This implies that $\delta = \alpha_1^2 / d_{11}$ is a square as well.  First,
from \eqref{eq:Dformula}, note that the coefficients of $\theta^{n-1}$ and $\theta^{n-2}$ of
$d_{11}$ are 0, since $a_{11}=b_{11}=0$. We start with the following
lemma:
\begin{lemma}\label{lemredfirst}
The coefficient of $\theta^{n-3}$ in $d_{11}$ is a square.
\end{lemma}
\begin{proof}
From \eqref{eq:Dformula} and the fact that $a_{11} = b_{11} = 0$,
the coefficient of $\theta^{n-3}$ in $d_{11}$ is equal to the
$11$-coefficient of the matrix $f_0(BA^{-1})^2A=f_0BA^{-1}B$.  Let $M$
denote the cofactor matrix of $A$, i.e., the $ij$-coefficient  $m_{ij}$
of $M$ is equal to $(-1)^{i+j}$ times the determinant of the matrix
obtained by removing the $i$th row and the $j$th column of $A$. Then
the coefficient of $\theta^{n-3}$ in $d_{11}$ is equal to the
$11$-coefficient of $BMB$.

We now describe the coefficients of $M$. Let $A^{{\rm top}}$ denote
the top-right $(n-1)/2,(n+1)/2$ submatrix of $A$. Note that, since $A$
is symmetric, the bottem-left $(n+1)/2,(n-1)/2$ submatrix of $A$ is
simply the transpose of $A^{{\rm top}}$. For $i\in [(n+1)/2,n]$ let
$A_i$ denote the $(n-1)/2,(n-1)/2$ matrix obtained by removing the
$i-(n-1)/2$'th column of $A^{{\rm top}}$. Then removing the
$i-(n-1)/2$'th row of the transpose of $A^{{\rm top}}$ yields $A_i^t$.
Since the top-right $(n-1)/2,(n-1)/2$ block of $A$ is 0, it follows
that for $i,j>(n-1)/2$ we have $m_{ij}=(-1)^{i+j}\Det(A_i)\Det(A_j)$.
Therefore, we have
\begin{equation*}
  \begin{array}{rcl}
    11\mbox{-coefficient of }BMB&=&
    \displaystyle\sum_{i,j=1}^nb_{1i}m_{ij}b_{j1}\\[.2in]
    &=&\displaystyle
    \sum_{i,j=(n+1)/2}^n (-1)^{i+j}b_{1i}b_{1j}\det A_i\det A_j\\[.2in]
    &=&\displaystyle\left(\sum_{k=(n+1)/2}^n (-1)^{k}b_{1k}\det A_k\right)^2,
  \end{array}
\end{equation*}
as necessary.
\end{proof}
Next, we show that the constant coefficient of $d_{11}$ (considered as
a polynomial in $\theta$) is a square.
\begin{lemma}\label{lemredsecond}
The constant coefficient $d_{11}(0)$ of $d_{11}(\theta)$ is a square.
\end{lemma}
\begin{proof}
Because the binary $n$-ic invariant of $(A,B)$ is $f$, we have
$\det(Ax-By)=\det(Ix-BA^{-1}y)\det(A)=f(x,y)$.
Since $BA^{-1}$ satisfies its characteristic polynomial, we obtain
$$
\sum_{j=0}^{n}f_{n-j}(BA^{-1})^j=0.
$$
By \eqref{eq:Dformula}, we compute $d_{11}(0)$ to be the $11$-coefficient of the matrix
\begin{align*}
\displaystyle\Bigl(\sum_{j=0}^{n-1}f_{n-j-1}(BA^{-1})^j\Bigr)A
&=\displaystyle\Bigl(\sum_{j=0}^{n-1}f_{n-(j+1)}(BA^{-1})^{j+1}\Bigr)AB^{-1}A\\
&=\displaystyle\Bigl(\sum_{j=0}^{n}f_{n-j}(BA^{-1})^{j}\Bigr)AB^{-1}A-f_nAB^{-1}A. \\
&= -f_n A B^{-1} A
\end{align*}
Note that $B$ is invertible because $\det B = f_n \neq 0$ since $f$ is
irreducible.  The lemma now follows from the proof of Lemma
\ref{lemredfirst} and symmetry (and the fact that $n$ is odd).
\end{proof}

We next show that $d_{11}(m)$ is a square for every integer $m$,
by applying Lemma \ref{lemredsecond} on the pair $(A,B-mA)$. Let
$g$ denote the binary $n$-ic invariant of the pair $(A,B-mA)$, and let $g_k$
denote the coefficient of $x^{n-k}y^k$ in $g(x,y)$. We have 
$$g(x,y)=\det(Ax-(B-mA)y)=\det(A(x+my)-By)=f(x+my,y).$$
As a consequence, we compute the $g_k$ to be
$$g_k=\sum_{j=0}^k\binom{n-j}{k-j}f_jm^{k-j}.$$ By applying Lemma
\ref{lemredsecond} to $(A,B-mA)$, we see that the $11$-coefficient of
the following matrix is a square:

\begin{align}
&\displaystyle\left(\sum_{j=0}^{n-1}g_{n-j-1}(BA^{-1}-mI)^j\right)A \nonumber \\
=&\displaystyle\left(\sum_{k=0}^{n-1}g_{k}(BA^{-1}-mI)^{n-k-1}\right)A \nonumber \\
=&\displaystyle\left(\sum_{k=0}^{n-1}\biggl(\sum_{j=0}^k\binom{n-j}{k-j}f_jm^{k-j}
\biggr)(BA^{-1}-mI)^{n-k-1}\right)A \nonumber \\
=&\displaystyle\left(\sum_{k=0}^{n-1}\biggl(\sum_{j=0}^k\binom{n-j}{k-j}f_jm^{k-j}
\biggr)
\biggl(\sum_{i=0}^{n-k-1}(-1)^{k+i}\binom{n-k-1}{i}(BA^{-1})^im^{n-k-i-1}\biggr)
\right)A \nonumber \\
=&\displaystyle\left(\sum_{k=0}^{n-1}\sum_{j=0}^k\sum_{i=0}^{n-k-1}
(-1)^{i+k}f_j(BA^{-1})^{i}m^{n-i-j-1}\binom{n-j}{k-j}\binom{n-k-1}{i}\right) A \nonumber \\
=&\displaystyle\left(\sum_{i,j\geq 0}^{i+j\leq n-1}f_j(BA^{-1})^im^{n-i-j-1}
\sum_{k=j}^{n-i-1}(-1)^{k+i}\binom{n-j}{k-j}\binom{n-k-1}{i}\right)A \nonumber \\
=&\displaystyle\Biggl(\sum_{i,j\geq 0}^{i+j\leq n-1}f_j(BA^{-1})^im^{n-i-j-1}\Biggr)A, \label{eq:11square}
\end{align}
where the last equality is a consequence of the following lemma:
\begin{lemma}
For nonnegative integers $n$, $i$, and $j$ satisfying $i+j\leq n-1$, we have
$$
\sum_{k=j}^{n-i-1}(-1)^{k+i}\binom{n-j}{k-j}\binom{n-k-1}{i}=(-1)^{n+1}.
$$
\end{lemma}
\begin{proof}
By taking the $i$th derivative of both sides of the identity
\begin{equation*}
\frac{(1+x)^{n-j}-1}{x}=\sum_{k=j}^{n-1}x^{n-k-1}\binom{n-j}{n-k}
\end{equation*}
and setting $x=-1$, we obtain the lemma.
\end{proof}
Comparing the formulas \eqref{eq:11square} and \eqref{eq:Dformula}
with $\theta = m$ shows that $d_{11}(m)$ is a square for any integer
$m$. It is a classical result that a polynomial that takes only square
values on integers must itself be a square. We include a proof for
completeness.
\begin{lemma}
Suppose $f(x)\in\Z[x]$ takes square values at every integer. Then
$f(x)=g(x)^2$ for some integer polynomial $g(x)$.
\end{lemma}
\begin{proof}
Suppose for the sake of contradiction that $f(x)$ is a nonconstant
squarefree polynomial. Then the resultant $R(f,f')$ of $f$ and its
derivative is a nonzero constant. Choose a prime $p$ such that $p\nmid
R(f,f')$ and such that $p\mid f(n)$ for some integer $n$; such a prime
$p$ exists since there exist infinitely many primes dividing some
value of $f$ applied to integers. We have that $p\mid f(n+p)$ also. By the
assumption that $f$ takes square values, we also have that $p^2$
divides both $f(n)$ and $f(n+p)$. However, because $f(n+p)\equiv
f(n)+pf'(n)\pmod{p^2}$, we find that $p\mid f'(n)$ and thus $p\mid R(f,f')$,
yielding a contradiction.
\end{proof}
Thus it follows that the
$11$-coefficient of $D$ is a square, concluding the proof of
Theorem~\ref{threducible}.
\end{proof}

\begin{remark}
Theorem \ref{threducible} also follows from a different interpretation
of orbits of $V(\Q)$ in terms of Jacobians of hyperelliptic curves,
found in Wang's dissertation \cite{jerry-thesis}.
\end{remark}

For an order $\cO$, let $\cI_2(\cO)$ denote the $2$-torsion subgroup of
the ideal group of $\cO$, i.e., the group of invertible fractional
ideals $I$ of $\cO$ such that $I^2 = \cO$. Note that the group
$\cI_2(\cO)$ is trivial when $\cO$ is maximal. We have the following
result parametrizing elements of $\cI_2(\cO)$ for all primitive orders
$\cO$ arising from integral binary $n$-ic forms.
\begin{prop}\label{reducible}
Let $\cO_f$ be an order corresponding to the integral primitive irreducible and nondegenerate binary
$n$-ic form $f$. Then $\cI_2(\cO_f)$ is in natural bijection with the set
of projective reducible $\SL_n(\Z)$-orbits on $V(\Z)\cap\pi^{-1}(f)$.
\end{prop}
\begin{proof}
Theorem \ref{threducible} shows that a projective
$\SL_n(\Z)$-orbit on $V(\Z)$ corresponding to the pair $(\II,\delta)$ is
reducible exactly when $\delta$ is a square, say
$\delta=\kappa^2$. The map from projective reducible
$\SL_n(\Z)$-orbits on $V(\Z)\cap\pi^{-1}(f)$ to $\cI_2(R)$ that sends
such an orbit to $\kappa^{-1}\II\cdot\If^{-\frac{n-3}{2}}$ is clearly a
bijection.
\end{proof}

\subsection{Parametrizations over other rings} \label{sec:orbitstab}

Let $T$ be a principal ideal domain. We now describe an analogue of Theorem \ref{thm:paramZ} over $T$, and we study a rigidified version of the parametrization to better understand the orbits and stabilizers of the group action.

The following theorem describes how $\SL_n(T)$-orbits of $V(T)$ are related to rank $n$ rings and ideal classes; it is a restatement of \cite[Thm.~6.3]{melanie-2nn}, using the fact that our base ring $T$ is a principal ideal domain:

\begin{thm}[\cite{melanie-2nn}]\label{thm:bij} 
Let $f \in U(T)$ be a nondegenerate primitive binary $n$-ic form.
Then there is a bijection between $\SL_n(T)$-orbits of $(A,B) \in V(T)$ with $f_{(A,B)} = f$ and
equivalence classes of pairs $(\II,\delta)$ where $\II \subset K_f :=
T[x]/(f(x,1))$ is an ideal of $R_f$ and $\delta \in K_f^\times$
satisfying $\II^2 \subset \delta \If^{n-3}$ as ideals and $\Norm(\II)^2 =
\Norm(\delta) \Norm(\If^{n-3})$.
Two pairs $(\II,\delta)$ and $(\II',\delta')$ are equivalent if there exists $\kappa \in K_f^\times$ such that $\II' = \kappa \II$ and $\delta' = \kappa^2 \delta$.
\end{thm}

Note that in \cite[\S 6]{melanie-2nn} the theorems are stated for
$\SL_n^{\pm}(T)$-orbits instead of $\SL_n(T)$-orbits, where
$\SL_n^{\pm}(T)$ denotes the elements of determinant $\pm 1$ in
$\GL_n(T)$. However, since $n$ is odd here, we have $\SL_n^\pm(T) \cong \{ \pm 1 \}
\times \SL_n(T)$, and since $-1$ acts trivially on pairs $(A,B)$ by
\eqref{eq:GLnaction}, the $\SL_n(T)$-orbits are precisely the same as
the $\SL_n^{\pm}(T)$-orbits.

In order to understand the stabilizer of the action of $\SL_n(T)$ on
an element $(A,B) \in V(T)$, we now discuss precisely with what the
elements (instead of $\SL_n(T)$-orbits) of $V(T)$ are in
correspondence, in terms of the pair $(\II,\delta)$ along with a basis
for $\II$.

\begin{prop}[\cite{melanie-2nn}]\label{prop:basedbij}
Let $f \in U(T)$ be a nondegenerate primitive binary $n$-ic form. 
Let $K_f := T[x]/(f(x,1))$. Then the
nonzero elements $(A,B) \in V(T)$ with $f_{(A,B)} = f$ are in
bijection with equivalence classes of triples $(\II,\mathcal{B},\delta)$
where $\II \subset K_f$ is a based ideal of $R_f$, with an ordered basis
given by an isomorphism $\mathcal{B}: \II \to T^n$ of $T$-modules, and
$\delta \in K_f^\times$, satisfying $\II^2 \subset \delta \If^{n-3}$ as
ideals and $\Norm(\II)^2 = \Norm(\delta)\Norm(\If^{n-3})$.
Two such triples $(\II,\mathcal{B},\delta)$ and
$(\II',\mathcal{B}',\delta')$ are equivalent if and only if there exists
$\kappa \in K_f^\times$ such that $\II' = \kappa \II$, $\mathcal{B} \circ
(\times \kappa) = \mathcal{B}'$, and $\delta' = \kappa^2 \delta$.
\end{prop}

As stated, Proposition \ref{prop:basedbij} is a ``symmetric'' version
of the first part of \cite[Thm. 6.1]{melanie-2nn}.
For any $(A,B) \in V(T)$ corresponding to $(\II,\mathcal{B},\delta)$ in
Proposition \ref{prop:basedbij}, the action of $\SL_n(T)$ on $(A,B)$
as in  \eqref{eq:GLnaction} induces an action of $\SL_n(T)$ on
the basis $\mathcal{B}$ through the correspondence, namely as given in
\eqref{SLaction}. This action of $\SL_n(T)$ takes $\II$ to
itself and does not affect $\delta$, so $\SL_n(T)$ acts on the triples
$(\II,\mathcal{B},\delta)$. 
Quotienting both sides of the correspondence in Proposition
\ref{prop:basedbij} by $\SL_n(T)$ yields precisely Theorem \ref{thm:bij}.

For the computations in later sections, we are interested in the
stabilizer of $(A,B) \in V(T)$ in $\SL_n(T)$. Any $g \in \SL_n(T)$
that fixes $(A,B)$ must correspond to an automorphism of the
corresponding triple $(\II,\mathcal{B},\delta)$; as $g$ preserves the map
$\mathcal{B}$, it is, up to scaling, an automorphism of $I$ as a 
$\Z[T]$-module. Because the discriminant of the corresponding form $f$ is nonzero,
such a module homomorphism is given by multiplication by a nonzero scalar. Since $g$ also fixes $\delta$, 
in fact $g$ corresponds to multiplication by an element $\kappa \in K_f^\times$
with $\kappa^2 = 1$ (in fact, such $\kappa$ lie in $R_f^\times$).
Furthermore, since multiplication on $\mathcal{B}$ by $\kappa$ exactly
corresponds to multiplication by the matrix $g$, we must have
$\Norm(\kappa) = \det(g)= 1$.  It is also easy to check that any such
$\kappa$ yields an element $g\in\SL_n(T)$ that stabilizes $(A,B)$.
We thus have the following description of the stabilizers:

\begin{cor} \label{cor:stabs}
Fix a principal ideal domain $T$. Let $(A,B) \in V(T)$ be a
nondegenerate element with primitive binary $n$-ic invariant $f$, corresponding to the ring $R_f$
and the pair $(\II, \delta)$ under Theorem \ref{thm:bij}. Then the stabilizer
group in $\SL_n(T)$ of $(A,B)$ corresponds to the norm $1$ elements $R_f^\times[2]_{\Norm \equiv 1}$
of the $2$-torsion in $R_f^\times$.
\end{cor}

In the cases where $T$ is a field or $\Zp$, we may also describe the
$\SL_n(T)$-orbits of $V(T)$ corresponding to a given binary $n$-ic invariant in a simple way.
We restrict to {\em projective} orbits, i.e., those corresponding to $(\II,\delta)$ where
$\II$ is projective as an $R_f$-module. (In the case where $T$ is a field, this will
be no restriction.)

\begin{cor} \label{cor:orbits}
Let $T$ be a field or $\Zp$. Let $f$ be a separable nondegenerate binary
$n$-ic form with coefficients in $T$.  Then the projective $\SL_n(T)$-orbits of $V(T)$
with invariant binary $n$-ic form $f$ are in bijection with elements
of $(R_f^\times / (R_f^\times)^2)_{\Norm \equiv 1}$.
\end{cor}

\begin{proof}
Let $T=k$ be a field and let $f$ be a separable nondegenerate binary
$n$-ic form over $k$. Then
$R_f$ is a commutative $k$-algebra of dimension $n$, and
in particular, a direct product of field extensions of $k$ and thus a
principal ideal ring. It is easy to check that $\If=R_f$.
In this case, Theorem \ref{thm:bij} implies that $\SL_n(k)$-orbits on
$V(k)$ with binary $n$-ic invariant $f$ correspond to
equivalence classes of pairs $(\II, \delta)$, where $\II$ is a fractional
ideal of $R_f$ and $\delta \in R_f^\times$ such that $\II^2 = \delta
\If^{n-3} = \delta R_f$.  The only ideals in $R_f$ are products of
either the unit ideal or the zero ideal in each of the
factors; since $\delta$ must be invertible, we have $\II = R_f$ and
so $\Norm(\delta) = 1$.  Thus, the equivalence classes of the pairs
$(\II,\delta)$ are parametrized by norm $1$ elements $\delta$ of $R_f^\times /
(R_f^\times)^2$.

Now let $T = \Zp$.  The ring $R_f$ is a direct product of finite extensions of $\Zp$
and is thus a principal ideal ring.  For projective pairs $(\II,\delta)$ as in Theorem
\ref{thm:bij}, the norm condition implies that $\II^2 = \delta
\If^{n-3}$.  As a result, the ideal $\II$ is again determined by the element
$\delta$ of $R_f^\times$.  Furthermore, since $n-3$ is even, we obtain
that $$\Norm(\delta) = \left( \frac{\Norm(\II)}{\Norm(\If^{(n-3)/2})}
\right)^2$$ is a square, so the set of equivalence classes of pairs
$(\II,\delta)$ are parametrized by $(R_f^\times / (R_f^\times)^2)_{\Norm \equiv 1}$.
\end{proof}

\begin{example} \label{ex:paramreal}
For $k = \R$, for a given $f$ as above, we have that $R_f$ is
isomorphic to $\R^{r_1} \times \C^{r_2}$ for some nonnegative integers
$r_1$ and $r_2$ with $r_1 + 2 r_2 = n$.  Then the number of
$\SL_n(\R)$-orbits with invariant binary $n$-ic form $f$ is
$2^{r_1-1}$, and the order of the stabilizer in
$\SL_n(\R)$ is $2^{r_1+r_2-1}$.
\end{example}

\section{Counting binary $n$-ic forms in acceptable familes} \label{sec:binarynics}

Our goal in this section is to determine asymptotics for the number of
irreducible elements in {\em acceptable} families of binary $n$-ic forms
having bounded height, as well as to determine asymptotics for the number of irreducible
$\SL_2(\Z)$-orbits on $\SL_2(\Z)$-invariant acceptable families having
bounded Julia invariant. We first define an acceptable family of binary $n$-ic forms, as well as how to compute the size of such families when ordered by height. We then define the Julia invariant, and recall a result of \cite{BY-Julia} on the asymptotics of orbits of binary $n$-ic forms ordered by Julia invariant. 

\subsection{Acceptable families of binary $n$-ic forms}\label{acc}
Recall that $U(T)=\Sym_n(T^2)$ denotes the space of binary $n$-ic forms 
over a ring $T$, and an element $\gamma \in \SL_2(T)$ acts on $f \in U(T)$ via
$\gamma f(x,y)=f((x,y)\gamma)$. Let $\Delta(f)$ denote the discriminant
of a form $f \in U(T)$.  Let $U(\R)^{(r_2)}$ denote the
set of binary $n$-ic forms with coefficients in $\R$ that have nonzero
discriminant and $r_2$ pairs of complex conjugate roots for some fixed
$r_2\in\{0,\ldots,(n-1)/2\}$.
\begin{defn} For each finite prime $p$, let
$\Sigma_p\subset U(\Z_p) \setminus \{\Delta=0\}$ be a nonempty open
set whose boundary has measure $0$, and let $\Sigma_\infty=U(\R)^{(r_2)}$
for some such $r_2$.
 We say that a collection
$\Sigma=(\Sigma_p)_p \cup \Sigma_\infty$ is {\it acceptable} if, for all large enough
primes $p$, the set $\Sigma_p$ contains all elements $f\in U(\Z_p)$
with $p^2\nmid \Delta(f)$. We refer to each $\Sigma_\nu$ where $\nu$ is any finite or infinite place of $\bQ$ as a {\em local specification} of $\Sigma$ at $\nu$. To a collection $\Sigma$, we associate a
family $\U(\Sigma)$ of integral binary $n$-ic forms given by
\begin{equation*}
\U(\Sigma)=\{f\in U(\Z):\;f\in\Sigma_\nu\; \mbox{ for all places } \nu\},
\end{equation*}
and say that $\U(\Sigma)$ is {\it acceptable} if $\Sigma$ is. \end{defn}
 Note
that if $\Sigma_p$ is $\SL_2(\Z_p)$-invariant for every prime $p$ (the
set $\Sigma_\infty$ is automatically $\SL_2(\R)$-invariant), then
$\U(\Sigma)$ is $\SL_2(\Z)$-invariant. In this case, we say that such a collection $\Sigma$ is $\SL_2$-invariant. Additionally, for any $\U(\Sigma)$, note that there is a multi-subset $\Sigma_H = \{R_f \mid f \in \U(\Sigma)\}$ inside $\fR_H$. Similarly, for any $\SL_2$-invariant $\U(\Sigma)$, there is also a multi-subset $\Sigma_J = \{R_{[f]} \mid [f] \in \SL_2(\Z)\backslash\U(\Sigma)\}.$ We say that a family $\Sigma_H$ or $\Sigma_J$ is {\em acceptable} if it is defined by an acceptable family $\U(\Sigma)$ of integral binary $n$-ic forms. 

\subsection{Binary $n$-ic forms ordered by height}
In this subsection, we order real and integral binary $n$-ic forms
by the following height function:
\begin{equation}\label{eqHbnf}
H(f_0x^n+\cdots+f_ny^n):=\max|f_i|.
\end{equation}
For any subset $S$
of $U(\R)$ or $U(\Z)$, we denote the set of elements in $S$ having
height less than $X$ by $S_{H<X}$. For a subset $S$ of $U(\Z)$, we
denote the subset of irreducible elements in $S$ by $S^\irr$.
Asymptotics for the number of integral irreducible binary $n$-ic forms
having squarefree discriminant and bounded height is determined in
\cite{BSW-part2}. The key ingredient in that result is a tail estimate
on the number of integral binary $n$-ic forms having bounded height
whose discriminants are divisible by $p^2$ for large primes
$p$. Namely, let $W_p\subset U(\Z)$ denote the set of integral binary
$n$-ic forms with $p^2\mid\Delta(f)$. Then the following tail estimate
is proved in \cite{BSW-part2}:
\begin{prop}\label{thunifbcf}
We have
\begin{equation*}
\#\Bigl(\bigcup_{p>M}W_p\Bigr)_{H<X}=O\Bigl(\frac{X^{n+1}}{\sqrt{M}}\Bigr)
+o(X^{n+1}).
\end{equation*}
\end{prop}

The next theorem follows from Proposition \ref{thunifbcf} just as
\cite[Theorem 2.21]{arulmanjul-bqcount} follows from \cite[Theorem
  2.13]{arulmanjul-bqcount}.
\begin{thm}\label{thbcfHcount}
Let $\Sigma$ be an acceptable collection of local specifications. Then
we have
\begin{equation}\label{eqthbcfHcount}
\#\U(\Sigma)_{H<X}^\irr=\Vol(\Sigma_{\infty,H<X})
\prod_p\Vol(\Sigma_p)+o(X^{n+1}).
\end{equation}
\end{thm}
Note that since $\Vol(\Sigma_{\infty,H<X})$ grows like a nonzero
constant times $X^{n+1}$, the error term in the right hand side of
\eqref{eqthbcfHcount} is indeed smaller than the main term.

\subsection{$\SL_2(\Z)$-orbits on binary $n$-ic forms ordered 
by Julia invariant}\label{juliadef} 

Every binary $n$-ic form with real coefficients whose leading
coefficient $a_0$ is nonzero can be written as
\begin{equation*}
f(x,y)=a_0(x-\alpha_1y)\cdots(x-\alpha_ny),
\end{equation*}
with $\alpha_i\in\C$. For $t=(t_1,\ldots,t_n)\in\R^n$, consider the
positive definite binary quadratic form
\begin{equation*}
Q_t(x,y)=\sum_{i=1}^nt_i^2(x-\alpha_iy)(x-\bar{\alpha_i}y).
\end{equation*}
Work of Julia \cite{julia} and Stoll--Cremona \cite{cremonastoll-reductionbinaryforms} shows that if $t$ is chosen to minimize the
quantity
\begin{equation}\label{eqjulia}
\vartheta(f)=\frac{a_0^2|\Disc\,Q_t|^{n/2}}{t_1^2\cdots t_n^2},
\end{equation}
then $\vartheta$ is an $\SL_2(\R)$-invariant of $f$, i.e.,
$\vartheta(f)=\vartheta(\gamma\cdot f)$ for any $\gamma\in\SL_2(\R)$. We call
$\vartheta$ the {\it Julia invariant} of the binary $n$-ic form $f(x,y)$.
The Julia invariant is not a polynomial invariant, but it is
homogeneous of degree $2$, in the sense that $\vartheta(\lambda
f)=\lambda^2\vartheta(f)$ for $\lambda\in\R^\times$. Indeed, the roots of
$f$ and $\lambda f$ are the same; when we replace $f$
with $\lambda f$, the $a_0$ in the right hand side of \eqref{eqjulia} is
replaced with $\lambda^2 a_0$ while the remaining quantities stay the
same. In this section, we will order $\SL_2(\Z)$-orbits $[f]$ of $U(\Z)$ by
the degree $1$ invariant
\begin{equation}\label{eqJ}
J(f)=\sqrt{\vartheta(f)}.
\end{equation}
Note that we may define the Julia invariant for forms $f$ with leading coefficient $0$ by using an $\SL_2(\R)$-equivalent form with nonzero leading coefficient.

Asymptotics for the number of irreducible $\SL_2(\Z)$-orbits on
integral binary $n$-ic forms were recently computed by Bhargava and Yang
\cite{BY-Julia}. The following theorem is a rewording of \cite[Theorem 9]{BY-Julia}:

\begin{thm}\label{thjuliacount}
Let $n$ be a positive integer, and let $r_2 \in\{0,1,\ldots, \lfloor n/2 \rfloor \}$.  Let
$\Sigma$ be a collection of local specifications such that the family
$\U(\Sigma)$ is defined by finitely many congruence conditions, and
$\Sigma_\infty=U(\R)^{(r_2)}$. Then there exists a
constant $c_{n,{r_2}}$, depending only on $n$ and $r_2$, such that
\begin{equation}\label{eqJuliafinitecount}
\#(\SL_2(\Z)\backslash
\U(\Sigma)^{\irr}_{J<X})=c_{n,{r_2}}\prod_p\Vol(\Sigma_p)X^{n+1}+
O(X^{n+1-\frac2n}).
\end{equation}
\end{thm}

To prove Theorem \ref{thjuliacount}, the authors construct a
fundamental domain $F$ for the action of $\SL_2(\Z)$ on
$U(\R)^{(r_2)}$. This fundamental domain has the property that
$F_{J<X}=XF_{J<1}$. Estimating the number of irreducible integral
binary $n$-ic forms in $F_{J<X}$ is difficult because $F_{J<X}$ is not
compact and has a cusp going to infinity. Using an averaging
technique, they prove that the cuspidal region of $F_{J<X}$ contains
negligibly many irreducible integral binary $n$-ic forms, while
the non-cuspidal region has negligibly many reducible
binary $n$-ic forms. This allows them to prove that the left hand side
of \eqref{eqJuliafinitecount} is well approximated by the volume of
$F_{J<X}$, yielding the result. In fact, the constant $c_{n,k}$ in
Theorem \ref{thjuliacount} is simply $\Vol(F_{J<1})$. We now prove the
following theorem.
\begin{thm}\label{thbcfJcount}
Let $\Sigma$ be an acceptable $\SL_2$-invariant collection of local
specifications. Then we have
\begin{equation*}
\#(\SL_2(\Z)\backslash\U(\Sigma)_{J<X}^\irr)=
\Vol(\SL_2(\Z)\backslash\Sigma_{\infty,J<X})\prod_{p}\Vol(\Sigma_p)+o(X^{n+1}).
\end{equation*}
\end{thm}
\begin{proof}
For every $\epsilon>0$ there exists an acceptable collection
$(\Sigma'_\nu)_\nu$ such that $\Sigma_\infty=\Sigma'_\infty$,
$\Sigma_p\subset\Sigma'_p$ for each prime $p$,
$\prod_p\Vol(\Sigma_p)\geq\prod_p\Vol(\Sigma'_p)-\epsilon$, and the
set $\U(\Sigma')$ is defined by finitely many congruence
conditions. From Theorem \ref{thjuliacount}, we obtain
\begin{equation*}
\begin{array}{rcl}
\#(\SL_2(\Z)\backslash\U(\Sigma)_{J<X}^\irr)&\leq&
\#(\SL_2(\Z)\backslash\U(\Sigma')_{J<X}^\irr)\\[.1in]
&=&\displaystyle\Vol(\SL_2(\Z)\backslash\Sigma_{\infty,J<X})
\prod_{p}\Vol(\Sigma'_p)
+o(X^{n+1})\\[.2in]
&\leq&\displaystyle\Vol(\SL_2(\Z)\backslash\Sigma_{\infty,J<X})
(\prod_{p}\Vol(\Sigma_p)+\epsilon)
+o(X^{n+1}).
\end{array}
\end{equation*}
Letting $\epsilon$ tend to $0$, we obtain the required upper bound on
$\#(\SL_2(\Z)\backslash\U(\Sigma)_{J<X}^\irr)$.

To obtain the lower bound, we proceed as follows. For $\epsilon>0$, we
take sets $F_{J<1}^{(\epsilon)}$ to be a semi-algebraic bounded subset
of $F_{J<1}$ such that $\Vol(F_{J<1}^{(\epsilon)})\geq
(1-\epsilon)\Vol(F_{J<1})$. We denote $XF_{J<1}^{(\epsilon)}$ by
$F_{J<X}^{(\epsilon)}$. Just as \cite[Theorem 2.21]{arulmanjul-bqcount}
follows from \cite[Theorem 2.13]{arulmanjul-bqcount}, we obtain from
Proposition \ref{thunifbcf} the estimate
\begin{equation}\label{eqtempjc1}
\#(F_{J<X}^{(\epsilon)}\cap \U(\Sigma)^\irr)=
\Vol(F_{J<X}^{(\epsilon)})\prod_{p}\Vol(\Sigma_p)+o(X^{n+1}).
\end{equation}
From the proof of \cite[Theorem 9]{BY-Julia}, we have the following
estimate on the number of integral elements in the ``cuspidal
region'':
\begin{equation}\label{eqtempjc2}
\#((F_{J<X}\backslash F_{J<X}^{(\epsilon)})\cap \U(\Sigma)^\irr)
\leq \epsilon X^{n+1}+O(X^{n+1-\frac2n}).
\end{equation}
Combining \eqref{eqtempjc1} and \eqref{eqtempjc2} yields the required
lower bound on $\#(\SL_2(\Z)\backslash\U(\Sigma)_{J<X}^\irr)$ and
completes the proof of Theorem \ref{thbcfJcount}.
\end{proof}

\section{Counting orbits of pairs of $n\times n$ symmetric matrices}\label{sec:counting}
The main goal of this section is to determine asymptotics for the
number of irreducible $\SL_n(\Z)$-orbits of pairs of $n \times n$
symmetric matrices having bounded height and the number of irreducible
$\SL_2(\Z) \times \SL_n(\Z)$-orbits of pairs of $n \times n$ symmetric
matrices having bounded Julia invariant. We first construct
fundamental domains for the action of $\SL_n(\bZ)$ and $\SL_2(\bZ)
\times \SL_n(\bZ)$ on pairs of real $n\times n$ symmetric matrices. We
then show that the cusps of these fundamental domains have a
negligible number of irreducible integral points. Additionally, we
show that the number of reducible integral points in the main body of
these fundamental domains is also negligible. A theorem of Davenport
\cite{davenport-volume1} allows us to conclude that the number of
irreducible integral points of bounded height in the fundamental
domain for the action of $\SL_n(\bZ)$ or the number of irreducible
integer points of bounded Julia invariant in the fundamental domain
for the action of $\SL_2 \times \SL_n(\bZ)$ is asymptotically equal to
the volumes of their respective main bodies.

 Fix an odd integer $n\geq 3$ and let $m = (n-1)/2$. Recall that $V(T) =T^2 \otimes \Sym_2(T^n)$ is the
space of pairs of $n\times n$ symmetric matrices $(A,B)$ over a ring $T$. The group
$G(T) :=\SL_2(T) \times \SL_n(T)$ acts on $V(T)$ via the action
\begin{equation}\label{eqGaction}
(\gamma_2,\gamma_n)\cdot(A,B)=
(\gamma_n A\gamma_n^t,\gamma_n B\gamma_n^t)\gamma_2^t \quad \quad \textrm{ for all } (\gamma_2,\gamma_n) \in G(T).
\end{equation}
It is easy to verify that we have
\begin{equation}\label{eqpiinvar}
\pi((\gamma_2,\gamma_n)\cdot(A,B))=\gamma_2^*(\pi(A,B)) \quad \quad \textrm{ for all }  (\gamma_2,\gamma_n) \in G(T),
\end{equation}
where
\begin{equation*}
\Bigl(
\begin{array}{cc}a&b\\c&d\end{array}
\Bigr)^*:=
\Bigl(
\begin{array}{cc}a&-c\\b&-d\end{array}
\Bigr).
\end{equation*}
The space $V(\R)$ inherits a height function $H$ and Julia invariant $J$ via $\pi$:
\begin{align*}
H(A,B)&:=H(\pi(A,B))\\
J(A,B)&:=J(\pi(A,B))
\end{align*}
where $H$ and $J$ are defined on $U(\R)$ as in \S \ref{sec:binarynics}.  From
\eqref{eqpiinvar}, it follows that $H$ is $\SL_n(\R)$-invariant and
$J$ is $G(\R)$-invariant on $V(\R)$.

We say that an element $(A,B)\in V(\Z)$ with $\pi(A,B)=f$ is {\it
  absolutely irreducible} if
\begin{itemize}
\item[{\rm (1)}] $f$ corresponds an order in an $S_n$-field, and
\item[{\rm (2)}] $(A,B)$ is not reducible in the sense of Theorem
  \ref{threducible}.
\end{itemize}
We denote the set of absolutely irreducible elements in $V(\Z)$ by
$V(\Z)^\irr$.

\subsection{Construction of fundamental domains} \label{sec:funddomains}
For $0\leq r_2 \leq m=(n-1)/2$, recall that $U(\R)^{(r_2)}$ denotes
the set of binary $n$-ic forms in $U(\R)$ that have nonzero
discriminant and $r_2$ distinct pairs of complex conjugate roots in
$\PP^1(\C)$. Let
$V(\R)^{(r_2)}$ denote the set of elements in $V(\R)$ whose image
under $\pi$ lies in $U(\R)^{(r_2)}$.  In this subsection, we construct
fundamental domains for the actions of $\SL_n(\Z)$ and $G(\Z)$ on
$V(\R)^{(r_2)}$ for $0 \leq r_2 \leq m$.

\subsubsection*{Fundamental sets for the action of $\SL_n(\R)$ and $G(\R)$ on $V(\R)^{(r_2)}$} 

Fix an integer $r_2$ with $0\leq r_2\leq m$, and let $r_1=n-2r_2$. For
$f \in U(\R)^{(r_2)}$, the $\R$-algebra $R_f$
corresponding to $f$ is isomorphic to
$\R^{r_1}\times\C^{r_2}$. Corollary \ref{cor:orbits} states that
the $\SL_n(\R)$-orbits of $\pi^{-1}(f)$ are in bijection with elements
$\delta \in (R_f^\times/R_f^{\times 2})_{N\equiv 1}$, which in turn is
in natural bijection with the subset $\T(r_2) \subset \{\pm1\}^{r_1}\times\{1\}^{r_2}$ of elements having an even number of $-1$ factors
(independent of the choice of $f\in U(\R)^{(r_2)}$).  
For an element $\delta\in\T(r_2)$, let
$V(\R)^{(r_2),\delta}$ denote the set of $v\in V(\R)^{(r_2)}$
such that $v$ corresponds to the pair $(R_{\pi(v)},\delta)$ under the
bijection of Theorem \ref{thm:bij}. It
follows that for $f\in U(\R)^{(r_2)}$ and $\delta\in\T(r_2)$, the set
$\pi^{-1}(f)\cap V(\R)^{(r_2),\delta}$ consists of a single
$\SL_n(\R)$-orbit.

Therefore, to construct a fundamental domain for the action of
$\SL_n(\R)$ on $V(\R)^{(r_2),\delta}$, it is enough to pick one
element $v_f\in V(\R)^{(r_2),\delta}$ for each $f\in
U(\R)^{(r_2)}$. However, we require our fundamental set to be
semialgebraic in order to apply our geometry-of-numbers techniques.

Below, we give such a section $s_\delta: U(\R)^{(r_2)} \to V(\R)$ for general $\delta$, which will be necessary
for constructing the fundamental sets, but first we describe, for the case of $\delta = (1,1,\ldots, 1)$, the very pretty explicit section $e: U(T)\rightarrow
V(T)$ of $\pi$ for any ring $T$. When $T=\R$, it is easy to check that
$e(f)\in V(\R)^{(r_2),\delta}$ for $f\in U(\R)^{(r_2)}$.  For $n = 3$,
the section $e$ takes a binary cubic form $f(x,y) = f_0x^3 + f_1x^2y
+f_2xy^2 + f_3y^3$ to the pair
       $$\left(\left(\begin{array}{ccc}0 & 0 & 1 \\0 & -f_0 & 0 \\1 & 0 & -f_2\end{array}\right),\left(\begin{array}{ccc}0 & 1 & 0 \\1 & f_1 & 0 \\0 & 0 & f_3\end{array}\right)\right).$$
For $n = 5$, the map $e$ sends a binary quintic form $f(x,y) = f_0x^5 + f_1x^4y + f_2x^3y^2 + f_3x^2y^3 + f_4xy^4 + f_5y^5$ to
       $$\left(\left(\begin{array}{ccccc}0 & 0 & 0 & 0 & 1 \\0 & 0 & 0 & 1 & 0 \\0 & 0 & f_0 & 0 & 0 \\0 & 1 & 0 & f_2 & 0 \\1 & 0 & 0 & 0 & f_4\end{array}\right),\left(\begin{array}{ccccc}0 & 0 & 0 & 1 & 0 \\0 & 0 & 1 & 0 & 0 \\0 & 1 & -f_1 & 0 & 0 \\1 & 0 & 0 & -f_3 & 0 \\0 & 0 & 0 & 0 & -f_5\end{array}\right)\right).$$
For general $n$, a binary $n$-ic form $f(x,y) = f_0x^n + f_1x^{n-1}y+f_2x^{n-2}y^2 + \cdots + f_ny^n$ is mapped under $e$ to $((a_{ij}),(b_{ij}))$ where:
\begin{align*} 
&\bullet\ a_{k,n-k} = 1 \textrm{ for } 1 \leq k < \tfrac{n-1}{2} \textrm{ or } \tfrac{n-1}{2} < k < n
&& \bullet\ b_{k,n-1-k} = 1 \textrm{ for } 1 \leq k < n  \\
&\bullet\ a_{\frac{n-1}{2}+k,\frac{n-1}{2}+k} = (-1)^{\frac{n-1}{2}} f_{2k} \textrm{ for } 0 \leq k \leq \tfrac{n-1}{2}
&& \bullet\ b_{\frac{n-1}{2}+k,\frac{n-1}{2}+k} = (-1)^{\frac{n+1}{2}} f_{2k+1} \textrm{ for } 0 \leq k \leq \tfrac{n-1}{2} \nonumber \\
&\bullet\ a_{ij} = 0 \textrm{ otherwise } && \bullet\ b_{ij} = 0 \textrm{ otherwise.} \nonumber
\end{align*}

We now handle the case of general $\delta$. For a fixed $\delta \in
\T(r_2)$ and an element $f = f_0 x^n + \cdots + f_n y^n \in
U(\R)^{(r_2)}$ with $f_0\neq 0$, consider the pair
$(R_f,\delta)$. Given the basis $\langle
1,\theta,\ldots,\theta^{n-1}\rangle$ for $R_f$, the corresponding pair
$(A,B)$ may be written explicitly using \eqref{abstractvarphi} and
\eqref{explicitvarphi}. From the definitions of $\theta$ and $\delta$,
it follows that $\phi(\theta^i\otimes\theta^j)$ may be written as
polynomials of degree less than $n$ in $\theta$, whose coefficients
are polynomials in the $f_i$ and $1/f_0$. Since $\zeta_{n-2}$ and
$\zeta_{n-1}$ are polynomials in $\theta$ both with leading
coefficient $f_0$, the coefficients of $A$ and $B$ are polynomials in
the $f_i$ and $1/f_0$. We define the function
$s_\delta:U(\R)^{(r_2)}\to V(\R)$ by sending such a binary $n$-ic form
$f$ to this pair $(A,B)$.

We now have the following lemma:
\begin{lemma}
Let $S\subset U(\R)$ be a compact semialgebraic set that does not
contain $0$. Then there exists a finite subset $T\subset\SO_2(\R)$ and
semialgebraic subsets $S_\tt\subset S$ for each $\tt\in T$, such
that the leading coefficients of $\tt\cdot f$ are bounded away from
$0$ independent of $f\in S_\tt$, and that the union of the
$S_\tt$ is $S$.
\end{lemma}
\begin{proof}
The set $\tilde{S}=S\times\{(x,y):x^2+y^2=1\}\subset U(\R)\times\R^2$
is semialgebraic. The function $S\to\R_{\geq 0}$ given
by $$f\mapsto\max_{x^2+y^2=1}|f(x,y)|$$ is continuous and nonzero. Hence
its image is bounded away from $0$ by some $\epsilon>0$. Therefore, the set
\begin{equation*}
S_1:=\{(f,(x,y)):f\in S,\;(x,y)\in\R^2,\;x^2+y^2=1,\;|f(x,y)|>\epsilon/2\}
\end{equation*}
is semialgebraic and its projection to $S$ is all of $S$.  Given an
element $\lambda=(x,y)\in\R^2$ with $x^2+y^2=1$, let $S_\lambda$
denote the set of elements $f$ in $S$ such that $(f,\lambda)\in
S_1$. Since the projections of semialgebraic sets are semialgebraic,
it follows that $S_\lambda$ is semialgebraic. Since $S$ is compact,
and the $S_\lambda$ are open inside $S$, there exists a finite subset
$T'$ of $\{(x,y)\in\R^2:x^2+y^2=1\}$ such that the union of
$S_\lambda$ over all $\lambda$ in this finite set is $S$. Given
$\lambda=(x,y)$, choose $\tt\in\SO_2(\R)$ to be the matrix
$\Bigl(\begin{array}{cc}\cos t&-\sin t\\\sin t & \cos t
\end{array}\Bigr)$, where $\cos t=x$ and $\sin t=y$. 
The leading coefficient of $\tt\cdot f$ is $(\tt\cdot
f)(1,0)=f(x,y)>\epsilon/2$.  The lemma follows by taking $T$ to be
the finite set of matrices $\tt$ in $\SO_2(\R)$ corresponding to
the finite set $T'$ of pairs $\lambda=(x,y)$ in $\R^2$, and setting
$S_\tt$ to be $S_\lambda$, for $\tt$ corresponding to $\lambda$.
\end{proof}

We can clearly choose the sets $S_\tt$ to be disjoint in the above
lemma. The set $S=U(\R)_{H=1}$ satisfies the conditions of the above
lemma. For a fixed $r_2$, we may write $U(\R)^{(r_2)}_{H=1}$ as a
finite disjoint union of the sets $S_\tt^{(r_2)}=S_\tt\cap
U(\R)^{(r_2)}$.  We now take our fundamental set for the action of
$\SL_n(\R)$ on $V(\R)^{(r_2),\delta}$ to be the finite union
\begin{equation*}
\RR^{(r_2),\delta}_H:=\bigcup_{\tt}\R_{>0}\cdot
(\tt^*)^{-1}s_\delta(\tt\cdot S_\tt^{(r_2)}).
\end{equation*}

We define a fundamental set $\RR^{(r_2),\delta}_J$ for the action of
$G(\R)$ on $V(\R)^{(r_2),\delta}$ in exactly the same way by
considering the set $S=L_n$, where $L_n$ is constructed in
\cite[\S3]{BY-Julia} to be a semialgebraic bounded fundamental set for
the action of $\SL_2(\R)$ on the set of elements in $U(\R)$ having
Julia invariant $1$.

Let $\RR^{(r_2),\delta}_H(X)$ (resp., $\RR^{(r_2),\delta}_J(X)$)
denote the set of elements in $\RR^{(r_2),\delta}_H$
(resp., $\RR^{(r_2),\delta}_J$) having height (resp., Julia invariant)
bounded by $X$. The sets $(\tt^*)^{-1}s_\delta(\tt\cdot
S_\tt^{(r_2)})$ are bounded for $S=U(\R)_{H=1}$ and $S=L_n$ because
every $f\in \tt\cdot S_\tt$ has bounded coefficients and has
leading coefficient bounded away from $0$. Since both height
and Julia invariant on $V(\R)$ have degree $n$, the
coefficients of elements $(A,B)$ in $\RR_H^{(r_2),\delta}(X)$ and
$\RR_J^{(r_2),\delta}(X)$ are bounded by $O(X^{1/n})$, where the
implied constant is independent of $(A,B)$.

\subsubsection*{Fundamental domains for $\SL_n(\Z)\backslash \SL_n(\R)$ 
and $G(\Z)\backslash G(\R)$}

Let $\SL_n(\R)=N_nT_nK_n$ be the Iwasawa decomposition of $\SL_n(\R)$,
where $N_n\subset \SL_n(\R)$ denotes the set of unipotent lower
triangular matrices, $T_n\subset \SL_n(\R)$ denotes the set of
diagonal matrices, and $K_n =\rm{SO}_n(\R) \subset \SL_n(\R)$ is the
maximal compact subgroup. Let $\Si_H$ be a Siegel domain in
$\SL_n(\R)$ defined as
$$
\Si_H:=N_n'T_n'K_n,
$$ where $N_n'\subset N_n$ is the set of elements in $N_n$ whose
coefficients are bounded by $1$ in absolute value and $T_n'\subset
T_n$ is given by
$$ T_n':=\{{\rm
  diag}(t_1^{-1},t_2^{-1},\ldots,t_{n}^{-1}):t_1/t_2>c,\ldots,t_{n-1}/t_{n}>c\},
$$ for some constant $c>0$ that is sufficiently small to ensure the
existence of a fundamental domain $\FF_H$ for the action of $\SL_n(\Z)$
on $\SL_n(\R)$ that is contained in $\Si_H$.

Next, we pick $N_2'\subset N$ to be the set of elements whose
coefficients are bounded by $1$ in absolute value and $T_2'\subset
T_2$ to be the set
$$ T_2':=\{{\rm
  diag}(t^{-1},t):t>1/4\}.
$$
Let
$$
\Si_J:=(N_2',N_n')(T_2',T_n')(K_2,K_n)
$$ be a Siegel domain. Then $\Si_J$ contains a fundamental domain
$\FF_J$ for the action of $G(\Z)$ on $G(\R)$.

\subsubsection*{Fundamental domains for the action 
of $\SL_n(\Z)$ on $V(\R)^{(r_2)}$} 

The size of the stabilizer in $\SL_n(\R)$ of $v\in
V(\R)^{(r_2,\delta)}$ can be computed from Corollary
\ref{cor:stabs}. This size depends only on $r_2$ and we denote it by
$\sigma(r_2)$. It is well known that the size of the stabilizer in
$\SL_2(\R)$ of a generic element $f\in U(\R)^{(r_2)}$ is $3$ if $n=3$
and $r_2=0$, and $1$ otherwise. It follows that the size of the
stabilizer in $G(\R)$ of a generic element in $V(\R)^{(r_2),\delta}$
is $\sigma'(r_2)$, where $\sigma'(r_2)=3\sigma(r_2)$ if $n=3$ and
$r_2=0$ and $\sigma'(r_2)=\sigma(r_2)$ otherwise. By arguments
identical to those in \cite[\S2.1]{arulmanjul-bqcount}, we see that
$\FF_H\cdot\RR_H^{(r_2),\delta}$ is a $\sigma(r_2)$-fold cover of a
fundamental domain for the action of $\SL_n(\Z)$ on
$V(\R)^{(r_2),\delta}$ and that $\FF_J\cdot\RR_J^{(r_2),\delta}$ is a
$\sigma'(r_2)$-fold cover of a fundamental domain for the action of
$G(\Z)$ on $V(\R)^{(r_2),\delta}$, where
$\FF_H\cdot\RR_H^{(r_2),\delta}$ and $\FF_J\cdot\RR_J^{(r_2),\delta}$
are regarded as multisets. More precisely, the $\SL_n(\Z)$-orbit of
any $v\in V(\R)^{(r_2),\delta}$ is represented
$\#\Stab_{\SL_n(\R)}(v)/\#\Stab_{\SL_n(\Z)}(v)$ times in
$\FF_H\cdot\RR_H^{(r_2),\delta}$, with the analogous statement also
holding for the multiset $\FF_J\cdot\RR_J^{(r_2),\delta}$.

For an $\SL_n(\Z)$-invariant set $S \subset
V(\Z)^{(r_2),\delta}:=V(\R)^{(r_2),\delta}\cap V(\Z)$, let $N_H(S;X)$
denote the number of absolutely irreducible $\SL_n(\Z)$-orbits on $S$
that have height bounded by $X$. For a $G(\Z)$-invariant set $S'
\subset V(\Z)^{(r_2),\delta}$, let $N_J(S';X)$ denote the number of
absolutely irreducible $G(\Z)$-orbits on $S'$ whose Julia invariant is
bounded by $X$.  Let $v\in V(\Z)$ be absolutely irreducible with
resolvent form $f$. Then $f$ corresponds to an order $\cO$ in an
$S_n$-number field and $\cO^\times[2]_{N\equiv 1}$ is
trivial. Furthermore, $f$ has trivial stabilizer in $\SL_2(\Z)$ since
$\Aut(\cO)$ is trivial. Therefore, $v$ has trivial stabilizer in
$\SL_n(\Z)$ and $G(\Z)$.  For any set $L\subset V(\Z)$, let $L^\irr$
denote the set of absolutely irreducible elements in $L$. Let
$\RR_H^{(r_2),\delta}(X)$ (respectively, $\RR_J^{(r_2),\delta}(X)$) denote
the set of elements in $\RR_H^{(r_2)}$ (resp., $\RR_J^{(r_2)}$) having
height (resp., Julia invariant) bounded by $X$. Then we have the
following:
\begin{prop}\label{propfdcount}
Let notation be as above. We have
\begin{equation}\label{eqnsx1}
\begin{array}{rcl}
\displaystyle N_H(S;X)&=&
\displaystyle\frac{1}{\sigma(r_2)}\#\{\FF_H\RR_H^{(r_2),\delta}(X)\cap S^\irr\};\\[.2in]
\displaystyle N_J(S';X)&=&
\displaystyle\frac{1}{\sigma'(r_2)}\#\{\FF_J\RR_J^{(r_2),\delta}(X)\cap S'^\irr\}.
\end{array}
\end{equation}
\end{prop}

\subsection{Averaging and cutting off the cusp}
Let $G_0$ (respectively, $G_0'$) be a bounded open nonempty $K_n$-invariant
(resp., $K_2\times K_n$-invariant) set in $\SL_n(\R)$
(resp., $G(\R)$).  We abuse notation and refer to Haar measures in
both groups $\SL_n(\R)$ and $G(\R)$ by $\dh$.  From Proposition
\ref{propfdcount} and by an argument identical to the proof of
\cite[Theorem 2.5]{arulmanjul-bqcount}, we obtain
\begin{equation}\label{eqavgtwo}
\begin{array}{rcl}
N_H(S;X)&=&
\displaystyle\frac{1}{\sigma(r_2)\Vol(G_0)}
\displaystyle\int_{h\in\FF_H}\#\{hG_0\cdot
\RR_H^{(r_2),\delta}(X)\cap S^\irr\} \dh, \textrm{ and}\\[.2in]
N_J(S';X)&=&
\displaystyle\frac{1}{\sigma'(r_2)\Vol(G_0')}
\displaystyle\int_{h\in\FF_J}\#\{hG_0'\cdot
\RR_J^{(r_2),\delta}(X)\cap S'^\irr\} \dh,
\end{array}
\end{equation}
where the volumes of $G_0$ and $G_0'$ are computed with respect to
$\dh$. We use \eqref{eqavgtwo} to define $N_H(S;X)$
(resp.\ $N_J(S';X)$) even when $S$ (resp.\ $S'$) is not
$\SL_n(\Z)$-invariant (resp.\ $G(\Z)$-invariant).

Let $\FF_H'\subset \FF_H$ and $\FF_J'\subset \FF_J$ denote the sets of
elements $\gamma\in\FF_H$ and $\gamma\in\FF_J$ such that
$|a_{11}(v)|<1$ for every element $v\in\gamma\cdot G_0
\RR_H^{(r_2),\delta}(X)$ and $v\in\gamma\cdot G_0'
\RR_J^{(r_2),\delta}(X)$, respectively. We will refer to the integrals
of the integrands in \eqref{eqavgtwo} over $\FF_H'$ and $\FF_J'$ as
the ``cuspidal'' part of the integral, and to the integrals over
$\FF_H\setminus\FF_H'$ and $\FF_J\setminus\FF_J'$ as the ``main
body'' of the integral.

\subsubsection*{Absolutely irreducible points in the cusp}

We will prove that the number of absolutely irreducible integral
points in the cusp is negligible:
\begin{prop}\label{propirredcusp}
We have
\begin{align*}
\displaystyle\int_{h\in\FF_H'}\#\{hG_0\cdot
\RR_H^{(r_2),\delta}(X)\cap V(\Z)^\irr\} \dh
&=O(X^{n+1-\frac{1}{n}}), \textrm{ and}\\
\displaystyle\int_{h\in\FF_J'}\#\{hG_0'\cdot
\RR_J^{(r_2),\delta}(X)\cap V(\Z)^\irr\} \dh&=O(X^{n+1-\frac{1}{n}}).
\end{align*}
\end{prop}

First, we list sufficient conditions to guarantee that an element
$(A,B)\in V(\Z)$ is not absolutely irreducible:
\begin{lemma}\label{lemredcond}
Let $(A,B)\in V(\Z)$ be such that all the variables in one of the following sets vanish:
\begin{enumerate}[label=\textnormal{(\alph*)}]
\item $\{a_{ij},b_{ij}:1\leq i\leq k,\;1\leq j\leq n-k\}$ for some $1\leq k\leq n-1$. \label{redcond1}
\item $\{a_{ij},b_{ij}:1\leq i,j\leq (n-1)/2\}$. \label{redcond2}
  \end{enumerate}
Then $(A,B)$ is not absolutely irreducible.
\end{lemma}

\begin{proof}
If $(A,B)$ satisfies Condition \ref{redcond1}, then it is easy to see
that the binary $n$-ic invariant of $(A,B)$ has a repeated factor over
$\Q$. Thus, the discriminant of the form vanishes. If $(A,B)$
satisfies Condition \ref{redcond2}, then clearly the quadratic forms
$A$ and $B$ have a common isotropic subspace of dimension
$(n-1)/2$. In either case, the pair $(A,B)$ is not absolutely
irreducible.
\end{proof}

Recall that the condition for $t=(t_1^{-1},\ldots,t_{n}^{-1})$ to be
an element of $T_n'$ is that $t_i/t_{i+1}>c$ for $1 \leq i \leq
n-1$. To simplify this condition, we use a change of variables: let
$s_i=t_i/t_{i+1}$ for $1\leq i\leq n-1$. Then $s=(s_1,\ldots,s_{n-1})$
is contained in $T'$ if and only if $s_i > c$ for each $i$.  The
action of the torus $T_2\times T_n$ of $G(\R)$ on $V(\R)$ multiplies
each coefficient by a monomial in $t,s_1,\ldots,s_{n-1}$. We denote
the set of coefficients of $V(\R)$ by $\var$; we have
\begin{equation*}\label{eqvar}
\var:=\{a_{ij},b_{ij}:1\leq i\leq j\leq n\}.
\end{equation*}
To each variable $c_{ij}$ in $\var$, we associate two weights: first,
the monomial $w_H(c_{ij})$ in the $s_i$ by which the action of $T_n$
scales $c_{ij}$ and second, the monomial $w_J(c_{ij})$ in $t$ and the
$s_i$ by which the action of $T_2\times T_n$ scales $c_{ij}$.  We
multiplicatively extend the function $w_H$ and $w_J$ to products of
integral powers of elements in $\var$.  We define a partial ordering
on $\var$ by setting $\alpha_1\lesssim_H \alpha_2$
(resp.\ $\alpha_1\lesssim_J \alpha_2$) whenever
$w_H(\alpha_2)/w_H(\alpha_1)$ (resp.\ $w_J(\alpha_2)/w_J(\alpha_1)$)
is a product of nonnegative powers of $s_i$ for each $i$ (resp.\ of
$t$ and $s_i$ for each $i$). The variable $a_{11}$ has minimal weight
under both these partial orderings. For a subset $\var'\subset\var$,
let $V(\Z)(\var')$ denote the set of $v \in V(\Z)$ such that
$\alpha(v) = 0$ for $\alpha \in \var'$. Then we have the following
immediate consequence of Lemma \ref{lemredcond}:
\begin{lemma}
Let $\var'\subset\var$ be a set that is closed under one of the
partial orderings $\lesssim_H$ and $\lesssim_J$. If
$V(\Z)(\var')^\irr$ is nonempty, then $\var'$ must be contained in the
following set:
\begin{equation*}
\var_0:=\{a_{ij}\in\var:i+j\leq n\}\cup \{b_{ij}\in\var:i+j\leq n-1\}\setminus\{b_{mm}\},
\end{equation*}
where $m = (n-1)/2$.
\end{lemma}

\begin{proof}[Proof of Proposition \ref{propirredcusp}]
By the arguments of \cite[\S3]{5sel}, it suffices to display the following
data in order to prove the part of Proposition \ref{propirredcusp}
regarding the height (resp.\ the Julia invariant): a function
$\psi:\var_0\setminus a_{11}\to\var\setminus\var_0$ such that
\begin{itemize}
\item[{\rm (1)}] 
$\alpha\lesssim_H\psi(\alpha)\;\forall
\alpha\in\var_0\backslash a_{11}$ (resp.\ 
$\alpha\lesssim_J\psi(\alpha)\;\forall
\alpha\in\var_0\backslash a_{11}$), and
\item[{\rm (2)}]
  $w_H\Bigl(\displaystyle\prod_{\alpha\in\var_0}\alpha^{-1}
\psi(\alpha)\Bigr)\cdot
  h_H$
  (resp.\ $w_J\Bigl(\displaystyle\prod_{\alpha\in\var_0}
\alpha^{-1}\psi(\alpha)\Bigr)\cdot
  h_J$) is a product of negative powers of the $s_i$ (resp.\ negative
  powers of $t$ and the $s_i$),
\end{itemize}
where $\psi(a_{11})$ is defined to be $1$, and where $h_H$ and $h_J$ are
factors arising from the Haar measures of $\SL_n(\R)$ and $G(\R)$ and are
given by
\begin{align*}
h_H := \displaystyle\prod_{k=1}^{n-1}s_k^{-nk(n-k)} && \textrm{and} &&
h_J := t^{-2}\displaystyle\prod_{k=1}^{n-1}s_k^{-nk(n-k)}.
\end{align*}

First note that such a function $\psi$ satisfying the required
conditions regarding the Julia invariant automatically satisfies the
required conditions regarding the height (since
$\alpha\lesssim_J\beta$ implies $\alpha\lesssim_H\beta$.) We define
$\psi$ as follows:
\begin{equation}\label{eqtheta}
\begin{array}{rcl}
\psi(a_{ij})&:=&\displaystyle\left\{
\begin{array}{rl}
a_{1n} &{\rm for}\;i=1;\\
a_{i(n-i+1)} &{\rm for}\;i>1\;{\rm and}\;j\neq m;\\
a_{(m+1)(m+1)} &{\rm for}\;i>1\;{\rm and}\;j= m;
\end{array}
\right.\\[.3in]
\psi(b_{ij})&:=&\displaystyle\left\{
\begin{array}{rl}
b_{j(n-j)} &{\rm for}\;j<m;\\
b_{mm} &{\rm for}\;j=m;\\
b_{(n-j-1)(j+1)} &{\rm for}\;j>m.
\end{array}
\right.
\end{array}
\end{equation}
The function $\psi$ clearly satisfies the first of the two required
conditions. From an elementary computation, we see that
\begin{equation*}
w_J\Bigl(\displaystyle\prod_{\alpha\in\var_0}
\alpha^{-1}\psi(\alpha)\Bigr)\cdot h_J=
t^{-1}\prod_{k=1}^ms_k^{-2k}\prod_{k=m+1}^{n-1}s_k^{-2(k-m)+1}.
\end{equation*}
This concludes the proof of Proposition \ref{propirredcusp}.
\end{proof}

\subsubsection*{Reducible points in the main body}
We say that an element $v\in V(\Z)$ is {\it bad} if $v$ is not
absolutely irreducible. Denote the set of bad elements in $V(\Z)$ by
$V(\Z)^\bad$.  We have the following theorem proving that the number
of bad elements in the main body is negligible.
\begin{prop}\label{propfewred}
We have
\begin{align*}
\displaystyle\int_{h\in\FF_H\backslash\FF_H'}\#\{hG_0\cdot
\RR_H^{(r_2),\delta}(X)\cap V(\Z)^\bad\} \dh
&=o(X^{n+1}), \textrm{ and} \\
\displaystyle\int_{h\in\FF_J\backslash\FF_J'}\#\{hG_0'\cdot
\RR_J^{(r_2),\delta}(X)\cap V(\Z)^\bad\} \dh
&=o(X^{n+1}).
\end{align*}
\end{prop}
\begin{proof}
For an integer $k$ with $2\leq k\leq n$, let $V(\Z)^{\neq k}$ denote
the set of elements $v\in V(\Z)$ such that, for each prime $p$, the
reduction modulo $p$ of the resolvent of $v$ does {\it not} factor
into a product of an irreducible degree $k$ factor and $n-k$ linear
factors. We claim that if the resolvent $f$ of an element $v\in V(\Z)$
does not correspond to an order in an $S_n$-field, then $v$ belongs to
$V(\Z)^{\neq k}$ for some $k$. Indeed, if $v$ lies in the complement
of $V(\Z)^{\neq n}$, then the reduction modulo $p$ of $f$ is
irreducible for some prime $p$, implying that $f$ is irreducible and
hence $R_f$ is an order. Furthermore, the Galois group of the Galois
closure of the fraction field of $R_f$ contains a $k$-cycle for each
$k$, implying that this Galois group is $S_n$.

Hence we may write
\begin{equation*}
V(\Z)^\bad=(\cup V(\Z)^{\neq k})\bigcup V(\Z)^\red
\end{equation*}
where $V(\Z)^\red$ denotes the set of elements that are reducible in
the sense of Theorem \ref{threducible}.

For each prime $p$, let $V(\Fp)^{=k}$ denote the set of elements whose
cubic resolvents factor into a product of a degree $k$ irreducible
factor and $n-k$ distinct linear factors.  Let $V(\Fp)^\irr$ denote
the set of elements in $v\in V(\Fp)$ such that every lift
$\tilde{v}\in V(\Z)$ is not reducible in the sense of Theorem
\ref{threducible}. Let $V(\Fp)^\nostab$ denote the set of elements
which have trivial stabilizer in $G(\Fp)$.  Then, from
\cite[\S3]{5sel}, it suffices to prove the following estimates:
\begin{align}
\#V(\Fp)^{=k} &\gg\#V(\Fp), \textrm{ and} \nonumber \\
\#V(\Fp)^\irr &\gg\#V(\Fp). \label{ineq:VFpestimate}
\end{align}
Let $U(\Fp)^{=k}$ denote the set of binary $n$-ic forms that factor
into a degree $k$ irreducible polynomial and $n-k$ distinct linear
factors. For every element $f\in U(\Fp)^{=k}$, the algebra $R_f$ is
isomorphic to a product of a degree $k$ extension of $\F_p$ and $n-k$
copies of $\Fp$. Therefore, the stabilizer in $\SL_n(\Fp)$ of every
element $v\in V(\Fp)^{=k}$ is independent of $v$ and $p$. Every lift
in $U(\F_p)^{=k}$ has at least one lift to $V(\Fp)^{=k}$
(corresponding to $\delta=1$). It follows that
\begin{equation*}
\#V(\Fp)^{=k}\gg\#U(\F_p)^{=k}\cdot\#\SL_n(\Fp)\gg\#V(\Fp),
\end{equation*}
as desired.

The proof of the inequality \eqref{ineq:VFpestimate} is similar. It follows from the
observation that every element in $V(\Fp)^{=n}$ that corresponds to a
nonidentity element in $\F_{p^n}^\times/(\F_{p^n}^\times)^2_{N\equiv
  1}$, under the bijection of Corollary \ref{cor:orbits}, belongs
to $V(\Fp)^\irr$.
\end{proof}

\subsubsection*{Absolutely irreducible points in the main body}

Let $L\subset V(\Z)$ be a lattice or a translate of a lattice in
$V(\R)$, and let $L^{(r_2),\delta}$ denote $L\cap
V(\Z)^{(r_2),\delta}$. We have already proved that the number of
irreducible integral points in the cusp is negligible and that the
number of reducible integral points in the main body is
negligible. Therefore, from \eqref{eqavgtwo}, Proposition
\ref{propirredcusp}, and Proposition \ref{propfewred}, we have
\begin{align*}
N_H(L^{(r_2),\delta},X)&=
\displaystyle\frac{1}{\sigma(r_2)\Vol(G_0)}
\displaystyle\int_{h\in\FF_H\setminus\FF_H'}\#\{hG_0\cdot
\RR_H^{(r_2),\delta}(X)\cap L\} \dh+o(X^{n+1}), \textrm{ and}\\
N_J(L^{(r_2),\delta},X)&=
\displaystyle\frac{1}{\sigma'(r_2)\Vol(G_0')}
\displaystyle\int_{h\in\FF_J\setminus\FF_J'}\#\{hG_0'\cdot
\RR_J^{(r_2),\delta}(X)\cap L\} \dh+o(X^{n+1}).
\end{align*}

To estimate the number of lattice points in $hG_0\cdot
\RR_H^{(r_2),\delta}(X)$ and $hG_0'\cdot
\RR_J^{(r_2),\delta}(X)$, we have the following result of
Davenport~\cite{davenport-volume1}.
\begin{prop}\label{davlem}
  Let $\mathcal R$ be a bounded, semi-algebraic multiset in $\R^n$
  having maximum multiplicity $m$, and that is defined by at most $k$
  polynomial inequalities each having degree at most $\ell$.  Then the
  number of integral lattice points $($counted with multiplicity$)$
  contained in the region $\mathcal R$ is
\[\Vol(\mathcal R)+ O(\max\{\Vol(\bar{\mathcal R}),1\}),\]
where $\Vol(\bar{\mathcal R})$ denotes the greatest $d$-dimensional
volume of any projection of $\mathcal R$ onto a coordinate subspace
obtained by equating $n-d$ coordinates to zero, where
$d$ takes all values from
$1$ to $n-1$.  The implied constant in the second summand depends
only on $n$, $m$, $k$, and $\ell$.
\end{prop}

The coefficient $a_{11}$ has minimal weight among all the
coefficients. Furthermore, for $h\in\FF_H\setminus\FF_H'$, the volume
of the projection of $hG_0\cdot \RR^{(r_2)}(X)$ onto the
$a_{11}$-coordinate is bounded away from $0$ by the definition of
$\FF_H'$.  Therefore, for $h\in\FF_H\setminus\FF_H'$, all proper
projections of $hG_0\cdot \RR^{(r_2)}(X)$ are bounded by a constant
times its projection onto the $a_{11}=0$ hyperplane. Proposition
\ref{davlem} thus implies that
\begingroup
\addtolength{\jot}{1em}
\begin{align*}
N_H(L^{(r_2),\delta},X)&= \displaystyle\frac{1}{\sigma(r_2)\Vol(G_0)}\displaystyle\int_{h\in(\FF_H\setminus\FF_H')}
\#\{hG_0\cdot \RR_H^{(r_2),\delta}(X)\cap L\} \dd h+o(X^{n+1}) \\
&= \displaystyle\frac{1}{\sigma(r_2)\Vol(G_0)}\displaystyle\int_{h\in(\FF\setminus\FF')}
\Vol_L(hG_0\cdot \RR_H^{(r_2),\delta}(X)) \dd h +o(X^{n+1}) \\
&= \displaystyle\frac{1}{\sigma(r_2)\Vol(G_0)}\Vol(\FF_H)\Vol_L(G_0\cdot \RR_H^{(r_2),\delta}(X))
+o(X^{n+1})\\
&=\displaystyle\frac{1}{\sigma(r_2)}\Vol_L(\FF_H\cdot \RR_H^{(r_2),\delta}(X))
+o(X^{n+1}),
\end{align*}
\endgroup
where the volume $\Vol_L$ of sets in $V(\R)$ is computed with respect
to the Euclidean measure on $V(\R)$ normalized so that $L$ has covolume
$1$, and where the third equality follows since $\Vol(\FF')$ tends to
zero as $X$ tends to infinity, and $\Vol_L(hG_0\cdot \RR^{(r_2)}(X))$
is independent of $h$, and the final equality follows from the
Jacobian change of variables in Theorem \ref{thmjac}.

An identical argument yields the analogous estimate for
$N_J(L^{(r_2),\delta},X)$. Let $L_p$ denote the closure of $L$ in
$V(\Z_p)$. Then for measurable sets $B$ in $V(\R)$, we have
\begin{equation*}
\Vol_L(B)=\Vol(B)\cdot\prod_p \Vol(L_p),
\end{equation*}
where $\Vol(B)$ is computed with respect to the Euclidean measure in
$V(\R)$ normalized so that $V(\Z)$ has covolume $1$, and the volumes
of $L\subset V(\Z_p)$ are computed with respect to the Haar measure on
$V(\Z_p)$ normalized so that $V(\Z_p)$ has volume $1$.
We thus have the following
theorem:
\begin{thm}\label{thmainzcount}
Let notation be as above. Then we have
\begin{align*}
N_H(L^{(r_2),\delta},X)&=
\displaystyle\frac{1}{\sigma(r_2)}
\Vol(\FF_H\cdot \RR_H^{(r_2),\delta}(X))\prod_p\Vol(L_p)
+o(X^{n+1}), \textrm{ and}\\
N_J(L^{(r_2),\delta},X)&=
\displaystyle\frac{1}{\sigma'(r_2)}
\Vol(\FF_J\cdot \RR_J^{(r_2),\delta}(X))\prod_p\Vol(L_p)
+o(X^{n+1}).
\end{align*}
\end{thm}

\begin{remark}
Using the Selberg sieve identically as in \cite[\S3]{JacobArulS5}, we
may improve the error term in Proposition \ref{propfewred}, and thus
in Theorem \ref{thmainzcount}, to $O(X^{n+1-\frac{1}{5n}})$. However,
this additional saving will not be necessary for the results in this paper.
\end{remark}

\section{Sieving to projective elements and acceptable sets} \label{sec:sieve}

In this section, we first determine asymptotics for $\SL_n(\Z)$-orbits
and $G(\Z)$-orbits on certain families having bounded height. Second, we
determine asymptotics for $\SL_n(\Z)$-orbits and $G(\Z)$-orbits on
acceptable sets conditional on a tail estimate. This tail estimate is
unknown for $n\geq 5$, but is known when $n=3$ (see \cite[Proposition
  23]{manjulcountquartic}). We begin by describing the very large and acceptable families we study. 

For each prime $p$, let $\Lambda_p\subset
V(\Z_p)\setminus\{\Delta=0\}$ be a nonempty open set whose boundary
has measure $0$. Let $\Lambda_\infty$ denote $V(\R)^{(r_2),\delta}$ for
some integer $r_2$ with $0\leq r_2\leq (n-1)/2$ and some
$\delta\in\{\pm1\}^{n-2r_2}\times\{1\}^{r_2}$.  To a collection
$\Lambda = (\Lambda_\nu)_\nu$ of these local specifications,
we associate the set
\begin{equation*}
\V(\Lambda):=\{v\in V(\Z):\;v\in\Lambda_\nu\ \textrm{ for all } \nu\}.
\end{equation*}

We say that the collection $\Lambda=(\Lambda_\nu)_\nu$ is {\it very
  large} (respectively,\ {\it acceptable}) if, for all large enough primes
$p$, the set $\Lambda_p$ contains all elements $v\in V(\Z_p)$ such
that $v$ is projective and the invariant form $f$ of $v$ is primitive,
i.e., the coefficients of $f$ are relatively prime
(resp., $p^2\nmid\Delta(v)$). We say that $\V(\Lambda)$ is {\it very
  large} or {\it acceptable} if $\Lambda$ is.

\subsection{Sieving to projective elements}
We define $V(\Z_p)^\proj$ to be the set of elements $(A,B)\in V(\Z_p)$
whose binary $n$-ic invariants are not divisible by $p$ and
correspond to a pair $(I,\delta)$ such that $I^2=(\delta)$. Then
 $$V(\Z)^{(r_2),\proj}=V(\Z)^{(r_2)}\bigcap\bigl(\bigcap_p
V(\Z_p)^\proj\bigr).$$ For a prime $p$, let $W_p$ now denote the set of
elements in $V(\Z)$ that do not belong to $V(\Z_p)^\proj$. We would like
to estimate the number of elements in $W_p$ for large $p$.  We have
the following theorem:
\begin{thm}\label{thunifproj}
We have
\begin{align*}N_H(\displaystyle\cup_{p\geq M} W_p,X)
&=O(X^{n+1}/M^{1-\epsilon}) +o(X^{n+1}), \textrm{ and} \\
N_J(\displaystyle\cup_{p\geq M} W_p,X)
&=O(X^{n+1}/M^{1-\epsilon}) +o(X^{n+1}),
\end{align*}
where the implied constant is independent of $X$ and $M$.
\end{thm}
\begin{proof}
If $(A,B) \in W_p$ gives rise to the binary $n$-ic form $f$, then the
ring $R_f$ is nonmaximal at $p$, which implies that
$p^2\mid\Delta(A,B)=\Delta(f)$.  Let $(A,B)\in W_p$, regarded as an
element of $V(\Z_p)$, correspond to a pair $(I,\delta)$ with $I^2\neq
(\delta)I_f^{n-3}$. Then the reduction of $(A,B)$ modulo $p$ corresponds to
the pair $(I\otimes\F_p,\overline{\delta})$, where $\overline{\delta}$
is the reduction of $\delta$ modulo $p$. From Nakayama's lemma, it
follows that $I^2\otimes\F_p\neq (\overline{\delta})I_f^{n-3}\otimes\F_p$.

Let $(A_1,B_1)\in V(\Z)$ be any element congruent to $(A,B)$ modulo
$p$. Denote the binary $n$-ic form associated to $(A_1,B_1)$ by
$f_1$. If $(A_1,B_1)$ corresponds to the pair $(I_1,\delta_1)$, then
it follows (again from Nakayama's lemma) that $I_1^2\neq
(\delta_1)I_{f_1}^{n-3}$. Thus $(A_1,B_1)\in W_p$.

Also, the set of elements in $W_p$ whose binary $n$-ic invariants are
divisible by $p$ is the preimage under $V(\Zp) \to V(\F_p)$ of the set
of elements in $V(\F_p)$ having binary $n$-ic invariant $0$.  It
follows that $W_p$ is defined via congruence conditions modulo $p$,
i.e., the set $W_p$ is the preimage of some subset of $V(\F_p)$ under
the reduction modulo $p$ map.

To prove the theorem, we start with the fundamental domain $\FF_H$
chosen in \S \ref{sec:funddomains}. For every $0<\epsilon<1$, we pick
a set $\FF^{(\epsilon)}\subset \FF_H$ which is open and bounded and
whose measure is $(1-\epsilon)$ times the measure of $\FF_H$.  Let
$\RR$ be the union of the $\RR^{(r_2),\delta}$ over all possible $r_2$
and $\delta$, and let $\RR_X$ denote the set of elements in $\RR$
having height bounded by $X$. Then, since the set
$\FF^{(\epsilon)}\cdot \RR_X$ is homogeneously expanding with $X$ and
since the reduction of the set $W_p$ modulo $p$ has codimension
greater than $2$ in $V(\F_p)$, we obtain
\begin{equation*}
\#\{\FF^{(\epsilon)}\cdot \RR_X\cap (\cup_{p\geq M} W_p)\}=O(X^{n+1}/M\log M)+O(X^n)
\end{equation*}
from an immediate application of \cite[Theorem 3.3]{manjul-geosieve}.
We further obtain
\begin{equation*}
\#\{(\FF\backslash\FF^{(\epsilon)})\cdot \RR_X\cap V(\Z)^{\irr}\}=O(\epsilon X^{n+1})
\end{equation*}
from the methods of the previous section. The first assertion of the
theorem follows. The second assertion follows in an identical fashion
by starting with $\FF_J$ instead of $\FF_H$.
\end{proof}

We now have the following theorem.

\begin{thm}\label{thsieveproj}
Let $r_2$ be an integer such that $0\leq r_2\leq (n-1)/2$ and let
$\delta\in\{\pm1\}^{n-2r_2}\times\{1\}^{r_2}$ be fixed.  Let $\Lambda$
be a very large collection of local specifications such that
$\Lambda_\infty=V(\R)^{(r_2),\delta}$. Then we have
\begin{align*}
N_H(\V(\Lambda),X)&=\displaystyle
\frac{1}{\sigma(r_2)}\Vol(\FF_H\cdot
\RR_H^{(r_2),\delta}(X))\prod_p \Vol(\Lambda_p)+o(X^{n+1}), \textrm{ and}\\
N_J(\V(\Lambda),X)&=\displaystyle
\frac{1}{\sigma'(r_2)}\Vol(\FF_J\cdot
\RR_J^{(r_2),\delta}(X))\prod_p \Vol(\Lambda_p)+o(X^{n+1}),
\end{align*}
where the volumes of sets in $V(\Z_p)$ are computed with respect to
the Euclidean measure normalized so that $V(\Z_p)$ has measure 1.
\end{thm}
\noindent The first estimate asserted by Theorem \ref{thsieveproj}
follows from Theorem \ref{thunifproj} just as \cite[Theorem
  2.21]{arulmanjul-bqcount} follows from \cite[Theorem
  2.13]{arulmanjul-bqcount}. The second estimate follows from a proof
identical to that of Theorem \ref{thbcfJcount} (which itself uses the
methods of the proof of \cite[Theorem 2.21]{arulmanjul-bqcount}).

\subsection{Sieving to acceptable sets (conditional on a tail 
estimate)} \label{sec:sfsieve}

Let $\Lambda$ be an acceptable collection of local specifications with
$\Lambda_\infty=V(\R)^{(r_2),\delta}$. Then we have the following
theorem whose proof is identical to the proof of the upper bound in 
\cite[Theorem 2.21]{arulmanjul-bqcount}:
\begin{thm}\label{thsieveupperbound}
We have
\begin{align*}
N_H(\V(\Lambda),X)&\leq\displaystyle
\frac{1}{\sigma(r_2)}\Vol(\FF_H\cdot \RR_H^{(r_2),\delta}(X))\prod_p
\Vol(\Lambda_p)+o(X^{n+1}), \textrm{ and}\\
N_J(\V(\Lambda),X)&\leq\displaystyle
\frac{1}{\sigma(r_2)}\Vol(\FF_J\cdot \RR_J^{(r_2),\delta}(X))\prod_p
\Vol(\Lambda_p)+o(X^{n+1}),
\end{align*}
where the volumes of sets in $V(\R)$ are computed with respect to
Euclidean measure normalized so that $V(\Z)$ has covolume $1$ and the
volumes of sets in $V(\Z_p)$ are computed with respect to the
Euclidean measure normalized so that $V(\Z_p)$ has volume 1.
\end{thm}

For a prime $p$, let $\W_p$ denote the set of elements in $V(\Z)$ such
that $p^2\mid\Delta$. The following estimates are unknown but likely
to be true:
\begin{equation}\label{eqassumedte}
\begin{array}{rcl}
N_H(\displaystyle\cup_{p\geq M} \W_p,X)
&=&O(X^{n+1}/M^{1-\epsilon}) +o(X^{n+1})\\[.1in]
N_J(\displaystyle\cup_{p\geq M} \W_p,X)
&=&O(X^{n+1}/M^{1-\epsilon}) +o(X^{n+1})
\end{array}
\end{equation}
We now have the following theorem.
\begin{thm}\label{thabcv}
Assume that one of the equations in \eqref{eqassumedte} holds.  Let
$\Lambda$ be an acceptable collection of local specifications with
$\Lambda_\infty=V(\R)^{(r_2),\delta}$. Then we have
\begin{align*}
N_H(\V(\Lambda),X)&=\displaystyle
\frac{1}{\sigma(r_2)}\Vol(\FF_H\cdot \RR_H^{(r_2),\delta}(X))\prod_p
\Vol(\Lambda_p)+o(X^{n+1}), \textrm{ and}\\
N_J(\V(\Lambda),X)&=\displaystyle
\frac{1}{\sigma(r_2)}\Vol(\FF_J\cdot \RR_J^{(r_2),\delta}(X))\prod_p
\Vol(\Lambda_p)+o(X^{n+1}),
\end{align*}
where the volumes of sets in $V(\R)$ are computed with respect to
Euclidean measure normalized so that $V(\Z)$ has covolume $1$ and the
volumes of sets in $V(\Z_p)$ are computed with respect to the
Euclidean measure normalized so that $V(\Z_p)$ has volume 1.
\end{thm}
\begin{proof}
We first assume that the first equation in \eqref{eqassumedte}
holds. Then the first assertion of the theorem follows just as
\cite[Theorem 2.21]{arulmanjul-bqcount} follows from \cite[Theorem
  2.13]{arulmanjul-bqcount}. The second estimate follows from a proof
identical to that of Theorem \ref{thbcfJcount}.

We now assume that the second equation in \eqref{eqassumedte}
holds. Then the second assertion of the theorem follows just as
\cite[Theorem 2.21]{arulmanjul-bqcount} follows from \cite[Theorem
  2.13]{arulmanjul-bqcount}.  To prove the first assertion, we use
methods from the proof of \cite[Lemma 3.7]{manjul-geosieve}. The set
$\FF_H\cdot \RR_H^{(r_2),\delta}(X)\setminus\{\Delta=0\}$ can be
covered with countably many fundamental domains for the action of
$G(\Z)$ on $V(\R)^{(r_2),\delta}$. Therefore, for any $\epsilon>0$,
there exist $s$ fundamental domains for the action of $G(\Z)$ on
$V(\R)^{(r_2),\delta}$ whose union covers all but measure $\epsilon
X^{n+1}$ of the finite measure multiset $\FF_H\cdot
\RR_H^{(r_2),\delta}(X)$, where $s$ is independent of $X$. (To ensure
that $s$ is independent of $X$, we merely choose $s$ fundamental
domains when $X=1$, and then scale these fundamental domains for large
$X$.) Once again arguments in the proof of \cite[Theorem
  2.21]{arulmanjul-bqcount} imply the bound
\begin{equation*}
\frac{N_H(\V(\Lambda),X)}{X^{n+1}}\geq 
\frac{1}{\sigma(r_2)}(\Vol(\FF_H\cdot \RR_H^{(r_2),\delta}(1))-\epsilon)
\prod_{p<M}
\Vol(\Lambda_p)+O(s/M^{1-\delta})+o(s).
\end{equation*}
Letting $M$ tend to $\infty$, and then $\epsilon$ to $0$, and then $s$
to $\infty$ yields the required lower bound. The upper bound follows
from Theorem \ref{thsieveupperbound}. This concludes the proof of
Theorem \ref{thabcv}.
\end{proof}

\section{Proof of the main theorems} \label{sec:proofmain}

We are now ready to prove Theorems \ref{thm:avgforfields}-\ref{thm:avgfororders}. To do so, we establish Theorem \ref{thmacceptablefam}, which determines an upper bound for the average sizes of the $2$-torsion subgroup in the class groups of acceptable families of orders of fixed signature ordered by height or by Julia invariant. For certain {\em very large} families, we obtain that the average sizes are in fact equal to 1; for all other acceptable families, the lower bound being equal to 1 is dependent on the tail estimates described in \eqref{eqassumedte}. The proof of Theorem \ref{thmacceptablefam} involves the computation of local volumes in order to determine the number of absolutely irreducible lattice points in $\FF_H$ of bounded height and $\FF_J$ of bounded Julia invariant. The results of \S \ref{sec:param} then allow us to conclude the theorem, and it immediately implies Theorems \ref{thm:avgforfields}, \ref{thm:avgforcubicfields}, and \ref{thm:avgfororders}. We obtain Theorem \ref{cor:oddclassno} from combining Theorems \ref{thm:avgforfields} and \ref{thm:avgforcubicfields} with the results of \cite{birchmerriman}.

We adopt the notation of the introduction. Recall that for an infinite
collection $\Sigma$ of local specifications, $\U(\Sigma)$ is the
associated set of integral binary $n$-ic forms, and acceptable sets
$\U(\Sigma)$ give rise to acceptable families $\Sigma_H \subseteq
\fR_H$ (and acceptable families $\Sigma_J \subseteq \fR_J$ if
$\U(\Sigma)$ is also $\SL_2(\Z)$-invariant). We now describe the
collections for which we obtain equalities on the average sizes in
Theorem \ref{thm:avgfororders}.

\begin{defn}\label{6.1} We say that $\Sigma = (\Sigma_\nu)_{\nu}$ and $\U(\Sigma)$ are {\em very large} if, for all sufficiently large primes $p$, the set
$\Sigma_p$ is precisely $U(\Z_p)\setminus pU(\Z_p)$. We say that a family $\Sigma_H\subseteq \fR_{H}$ is {\em very large} if it is defined by a very large family $\U(\Sigma)$, i.e., $\fR_H = \{R_f \mid f \in \U(\Sigma)\}.$ A family $\Sigma_J \subseteq \fR_J$ is {\em very large} if it is defined by a very large $\SL_2(\bZ)$-invariant family $\U(\Sigma)$. 
\end{defn}

\begin{thm}\label{thmacceptablefam}
Fix an integer $n$ and a signature $(r_1,r_2)$ with $r_1+2r_2=n$. Let
$\fR_1\subset\fR_H^{r_1,r_2}$ be a family of rings that arises from an acceptable
set of integral binary $n$-ic forms and let $\fR_2\subset\fR_J^{r_1,r_2}$ be a
family of rings that arises from an acceptable $\SL_2(\Z)$-invariant
set of binary $n$-ic forms. Then:
\begin{itemize}
\item[\rm (a)] The average sizes of
$$|\Cl_2(\cO)| - \frac{1}{2^{r_1+r_2-1}} |\mathcal{I}_2(\cO)|$$ over
  $\cO \in \fR_1$ ordered by height and over $\cO \in \fR_2$ ordered
  by Julia invariant are bounded above by 1.
  \item[\rm (b)] The average sizes of
$$|\Cl_2^+(\cO)| - \frac{1}{2^{r_2}} |\mathcal{I}_2(\cO)|$$ over $\cO
  \in \fR_1$ ordered by height and over $\cO \in \fR_2$ ordered by
  Julia invariant are bounded above by 1.
\end{itemize}
If we assume that $\fR_1$ and $\fR_2$ arise from very
large sets of binary $n$-ic forms, then the average sizes in {\rm(a)} and
{\rm (b)} are equal to $1$, independent of the choice of very large set. Furthermore, conditional on the tail estimates in \eqref{eqassumedte}, the average sizes in {\rm (a)} and {\rm (b)} are indeed equal to 1 for all $\fR_1$ or $\fR_2$ arising from any acceptable set of binary $n$-ic forms. 
\end{thm}

We will prove Theorem \ref{thmacceptablefam} in the following sections.

\subsection{Computing the product of local volumes}
We first prove a statement about the ``compatibility of measures''.
Let $\mathrm{d}v$ and $\mathrm{d}f$
denote Euclidean measures on $V$ and $U$, respectively, normalized so
that $V(\Z)$ and $U(\Z)$ have covolume $1$. Let $\omega$ be an
algebraic differential form that generates the rank $1$ module of
top degree left-invariant differential forms on $\SL_n$ over $\Z$.  We
have the following theorem, whose proof is identical to that of
\cite[Props.~3.11 \& 3.12]{arulmanjul-bqcount}.
\begin{thm}\label{thmjac}
  Let $T$ be $\R$, $\C$, or $\Z_p$ for some prime $p$.  Let $s:U(T)\to
  V(T)$ be a continuous section for $\pi$, i.e., a continuous
  function such that the invariant binary $n$-ic
  of $w_{f}:=s(f)$ is $f$.  Then there exists a rational nonzero
  constant $\J$ such that for any measurable function $\phi$ on
  $V(T)$, we have
\begin{align*}
    \int_{v \in \SL_n(T)\cdot s(U(T))}\phi(v)\dd v &=  |\J|\int_{U(T)}\displaystyle\int_{\SL_n(T)} \phi(g\cdot w_{f})\,\omega(g) \dd f \\[10pt]
    \int_{V(T)}\phi(v) \dd v&= |\J| \int_{\substack{f\in U(T)\\\Delta(f)\neq 0}}\Bigl(\sum_{v\in\textstyle{\frac{V(T)(f)}{\SL_n(T)}}}\frac{1}{|\Stab_{\SL_n(T)}(v)|}\int_{g\in\SL_n(T)}\phi(g\cdot v)\omega(g)\Bigr) \dd f
\end{align*}
where we regard $\SL_n(T)\cdot s(R)$ as a multiset and
$\frac{V(T)(f)}{\SL_n(T)}$ denotes a set of representatives for the
action of $\SL_n(T)$ on elements in $V(T)$ having invariant $f$.
\end{thm}

For $r_2 \in\{1,\ldots,(n-1)/2\}$ and for $f\in V(\Z_p)$ we define local masses
\begin{align*}
m_p(f) := 
\frac{|(R_f^\times/(R_f^\times)^2)_{N\equiv 1}|}{|R_f^\times[2]_{N\equiv 1}|}
&& \textrm{ and } && 
m_\infty(r_2) := \frac{\left|\left((\R^{n-2r_2} \times \C^{r_2})^\times/\left((\R^{n-2r_2} \times \C^{r_2})^\times\right)^2\right)_{N\equiv 1}\right|}{|(\R^{n-2r_2} \times \C^{r_2})^\times[2]_{N\equiv 1}|}.
\end{align*}
We denote the numerator and the denominator of the right hand side in
the equation defining $m_\infty(r_2)$ by $\tau(r_2)$ and
$\sigma(r_2)$, respectively. For a prime $p$, let $\Sigma_p\subset
U(\Z_p)\setminus pU(\Z_p)$ be a non-empty open set whose boundary
has measure $0$. Let $\Lambda_p$ denote the set of projective elements
in $V(\Z_p)$ whose invariant binary form belongs to $\Sigma_p$.
We have the following corollary to Theorem \ref{thmjac}:
\begin{cor}\label{corjac}
Let notation be as above. We have
\begin{align*}
\displaystyle\Vol(\FF_H\cdot \RR_H^{(r_2),\delta}(X)) &=
\displaystyle|\J|\Vol(\FF_H)\Vol(U(\R)^{(r_2)}_{H<X}),\\[.2in]
\displaystyle\Vol(\FF_J\cdot \RR_J^{(r_2),\delta}(X)) &=
\displaystyle\frac{\sigma'(r_2)}{\sigma(r_2)}|\J|\Vol(\FF_H)\Vol(\SL_2(\Z)\backslash U(\R)^{(r_2)}_{J<X}), \textrm{ and}\\
\displaystyle\Vol(\Lambda_p) &=\displaystyle|\J|_p \Vol(\SL_n(\Z_p)) 
\int_{f\in\Sigma_p}m_p(f) \dd f,
\end{align*}
where the volumes of $\FF_H$ and $\SL_n(\Z_p)$ are computed with
respect to $\omega$, and $\sigma'(r_2)$ denotes the size of the
stabilizer in $G(\R)$ of a generic element of $V(\R)^{(r_2)}$.
\end{cor}
\begin{proof}
The first equality follows immediately from Theorem
\ref{thmjac}. Next, note that we have $\FF_J=\FF_2\times\FF_H$, where
$\FF_2$ is a fundamental domain for the action of $\SL_2(\Z)$ on
$\SL_2(\R)$. Let the multiset $I\subset U(\R)$ denote the invariants
of the multiset $\FF_2\cdot \RR_J^{(r_2),\delta}(X)$. Then $I$
generically represents each element of $\SL_2(\Z)\backslash
U(\R)^{(r_2)}_{J<X}$ exactly $\sigma'(r_2)/\sigma(r_2)=s(r_2)$ times,
since $s(r_2)$ is the size of the stabilizer in $\SL_2(\R)$ of an
element in $U(\R)^{(r_2)}$. (We have already seen that $s(r_2)=3$ when
$n=3$ and $r_2=0$ and $s(r_2)=1$ otherwise.) The second equality now
follows immediately from Theorem~\ref{thmjac}.

To obtain the final equality, note that Theorem \ref{thmjac} implies
\begin{equation*}
   \int_{\Lambda_p} \dd v=
|\J|_p\Vol(\SL_n(\Z_p))\displaystyle\int_{f\in\Sigma_p}
\sum_{v\in\frac{\det^{-1}(f)}{\SL_n(\Z_p)}}\frac{1}{|\Stab_{\SL_n(\Z_p)}(v)|} \dd f,
\end{equation*}
where the sum runs over representatives in projective
$\SL_n(\Z_p)$-orbits of $\det^{-1}(f)$. The result now follows from
Corollary \ref{cor:stabs}.
\end{proof}
Denote $n-2r_2$ by $r_1$ so that $r_1+2r_2=n$. By Corollaries
\ref{cor:stabs} and \ref{cor:orbits} and Example \ref{ex:paramreal}, we have
\begin{equation} \label{eq:minfty}
\tau(r_2)=2^{r_1-1}, \quad \sigma(r_2)=2^{r_1+r_2-1}, \quad \textrm{and}  \quad m_\infty(r_2) = 2^{-r_2}.
\end{equation}
In \cite[Lemma 22]{BV3}, the values of $m_p(f)$ are computed for cubic
rings. We now compute these values for degree $n$ rings using a
similar argument.
\begin{lemma}\label{lemBK}
  Let $R$ be a nondegenerate ring of degree $n$ over $\Z_p$. Then
\begin{equation} \label{eq:norm1unitgroup}
\frac{|(R^\times/(R^\times)^2)_{N\equiv 1}|}{|R^\times[2]_{N\equiv 1}|}
\end{equation}
is $1$ if $p\neq 2$ and $2^{n-1}$ if $p=2$.
\end{lemma}

\begin{proof}
The unit group of $R^\times$ is the direct product of a finite abelian
subgroup and $\Zp^n$, and the norm $1$ part $R^\times_{N \equiv 1}$ is
also a direct product of a finite abelian group and $\Zp^{n-1}$. 
For $G$ a finite abelian group or $G = \Zp^n$ when $p \neq 2$, we have
$$\frac{|G/G^2|}{|G[2]|} = 1,$$
so the value of
\eqref{eq:norm1unitgroup} is $1$ for $p \neq 2$.  When $p = 2$,
because $2$ is not a unit in $\Z_2$, the $\Z_2$-module $2 \Z_2^{n-1}$
has index $2^{n-1}$ in $\Z_2^{n-1}$ instead, implying that
\eqref{eq:norm1unitgroup} evaluates to $2^{n-1}$.
\end{proof}

It follows that for a fixed prime $p$, the value of $m_p(f)$ is
independent of $f\in U(\Z_p)^\prim$. We denote this value by $m_p$.
We conclude with the following theorem:
\begin{thm}\label{thmprodvol}
We have
\begin{align*}
\displaystyle
\frac{1}{\sigma(r_2)}\Vol(\FF_H\cdot \RR_H^{(r_2),\delta}(X))\prod_p
\Vol(\Lambda_p)&=
2^{r_2}\displaystyle\Vol(U(\R)^{(r_2)}_{H<X})\prod_p\Vol(\Sigma_p) \textrm{ and}\\
\displaystyle\frac{1}{\sigma'(r_2)}
\Vol(\FF_J\cdot \RR_J^{(r_2),\delta}(X))\prod_p
\Vol(\Lambda_p)&=
2^{r_2}
\displaystyle\Vol(\SL_2(\Z)\backslash U(\R)^{(r_2)}_{J<X})\prod_p\Vol(\Sigma_p)
\end{align*}
\end{thm}

\begin{proof}
From Corollary \ref{corjac} and Lemma \ref{lemBK}, we obtain
\begin{align}
&\frac{1}{\sigma(r_2)}\Vol(\FF_H\cdot \RR_H^{(r_2),\delta}(X))\prod_p
\Vol(\Lambda_p) \nonumber \\
&\qquad \qquad = \frac{1}{\sigma(r_2)}|\J|\Vol(\FF_H)\Vol(U(\R)^{(r_2)}_{H<X})
\prod_p|\J|_p\Vol(\SL_n(\Z_p))m_p\Vol(\Sigma_p), \label{eq:prodvol1}
\end{align}
and
\begin{align}
&\frac{1}{\sigma'(r_2)} \Vol(\FF_J\cdot
  \RR_J^{(r_2),\delta}(X))\prod_p \Vol(\Lambda_p) \nonumber \\
 &\qquad \qquad =
\frac{1}{\sigma(r_2)}|\J|\Vol(\FF_H)\Vol(\SL_2(\Z)\backslash
  U(\R)^{(r_2)}_{J<X}) \prod_p|\J|_p\Vol(\SL_n(\Z_p))m_p\Vol(\Sigma_p). \label{eq:prodvol2}
\end{align}
We simplify the right hand side of these expressions by noting that
\begin{align}
 |\J|\prod_p|\J|_p&=1, \label{eq:Jacproduct}\\
 \Vol(\FF_H)\prod_p\Vol(\SL_n(\Z_p))&=1, \label{eq:Tamagawa}\\
 \frac{1}{\sigma(r_2)}\prod_pm_p&=2^{r_2}, \label{eq:mpprod}
\end{align}
where \eqref{eq:Jacproduct} follows from the product formula,
\eqref{eq:Tamagawa} comes from the Tamagawa number of
$\SL_n(\Q)$ being $1$, and \eqref{eq:mpprod} follows from
\eqref{eq:minfty} and Lemma \ref{lemBK}. Combining these
with \eqref{eq:prodvol1} and \eqref{eq:prodvol2} yields the theorem.
\end{proof}

\subsection{Proof of Theorem \ref{thmacceptablefam}}\label{proofofthm3}

Let $\fR\subset\fR_H$ be an acceptable family of rings having fixed
signature $(r_1,r_2)$. Then the rings in $\fR$ are in bijection with
an acceptable set $\U(\Sigma)\subset U(\Z)$ of binary $n$-ic forms
with $\Sigma_\infty=U(\R)^{(r_2)}$. Let $\Lambda^{(\delta)}$ be a
collection of local specifications for $V$, where $\Lambda_p$ consists
of projective elements in $V(\Z_p)$ whose invariants belong to
$\Sigma_p$ and $\Lambda_\infty=V(\R)^{(r_2),\delta}$. Then
$\Lambda=(\Lambda_\nu)_\nu$ is acceptable. Furthermore, if $\fR$ is
very large, then so is $\Lambda$.

From Propositions \ref{prophr} and \ref{reducible} and Lemma \ref{lem:sizeHplus}, we know that
\begin{align*}
\sum_{\substack{\cO\in\fR\\H(\cO)<X}} 2^{r_1+r_2-1}|\Cl_2(\cO)|-|\cI_2(\cO)|&=
\sum_\delta N_H(\V(\Lambda^{(\delta)}),X), \textrm{ and}\\
\sum_{\substack{\cO\in\fR\\H(\cO)<X}} 2^{r_2}|\Cl_2^+(\cO)|-|\cI_2(\cO)|&=
 N_H(\V(\Lambda^{(\delta_{\gg0})}),X),
\end{align*}
where the first sum is over all possible $\delta$ and $\delta_{\gg0}$
denotes the element $(1,1,\ldots,1)\in\R^{r_1}\times \C^{r_2}$. As a result,
we have
\begin{equation}\label{eqproofH}
\begin{array}{rcccl}
\displaystyle\lim_{X\rightarrow \infty}
\frac{\displaystyle\sum_{\substack{\cO\in\fR\\H(\cO)<X}} 2^{r_1+r_2-1}|\Cl_2(\cO)|-|\cI_2(\cO)|}
{\displaystyle\sum_{\substack{\cO\in\fR\\H(\cO)<X}}1}
&=&\displaystyle\lim_{X\rightarrow \infty}
\frac{\displaystyle\sum_\delta N_H(\V(\Lambda^{(\delta)}),X)}
{\#\U(\Sigma)_{H<X}}&\leq&2^{r_1+r_2-1}, \\[.5in]
\textrm{and }\quad  \displaystyle\lim_{X\rightarrow \infty}
\frac{\displaystyle\sum_{\substack{\cO\in\fR\\H(\cO)<X}} 2^{r_2}|\Cl_2^+(\cO)|-|\cI_2(\cO)|}
{\displaystyle\sum_{\substack{\cO\in\fR\\H(\cO)<X}}1}
&=&\displaystyle\lim_{X\rightarrow \infty}
\frac{\displaystyle N_H(\V(\Lambda^{(\delta_{\gg0})}),X)}
{\#\U(\Sigma)_{H<X}}&\leq&2^{r_2},
\end{array}
\end{equation}
where we use Theorems \ref{thsieveupperbound} and \ref{thbcfHcount} to
evaluate the numerators and the denominators of the middle terms in
the above equation, and Theorem \ref{thmprodvol} to evaluate the
product of local volumes that arise.

Similarly, let $\fR\subset\fR_J$ be an acceptable family of rings
having fixed signature $(r_1,r_2)$. Then the rings in $\fR$ are in
bijection with $\SL_2(\Z)$-orbits on an acceptable set
$\U(\Sigma)\subset U(\Z)$ of binary $n$-ic forms with
$\Sigma_\infty=U(\R)^{(r_2)}$. We define $\Lambda^{(\delta)}$ as
above, and obtain

\begin{equation}\label{eqproofJ}
\begin{array}{rcccl}
\displaystyle\lim_{X\rightarrow \infty}
\frac{\displaystyle\sum_{\substack{\cO\in\fR\\J(\cO)<X}} 2^{r_1+r_2-1}|\Cl_2(\cO)|-|\cI_2(\cO)|}
{\displaystyle\sum_{\substack{\cO\in\fR\\J(\cO)<X}}1}
&=&\displaystyle\lim_{X\rightarrow \infty}
\frac{\displaystyle\sum_\delta N_J(\V(\Lambda^{(\delta)}),X)}
{\#\SL_2(\Z)\backslash\U(\Sigma)_{J<X}}&\leq&2^{r_1+r_2-1}, \\[.5in]
\textrm{and} \quad \displaystyle\lim_{X\rightarrow \infty}
\frac{\displaystyle\sum_{\substack{\cO\in\fR\\J(\cO)<X}} 
2^{r_2}|\Cl_2^+(\cO)|-|\cI_2(\cO)|}
{\displaystyle\sum_{\substack{\cO\in\fR\\J(\cO)<X}}1}
&=&\displaystyle\lim_{X\rightarrow \infty}
\frac{\displaystyle N_J(\V(\Lambda^{(\delta_{\gg0})}),X)}
{\#\SL_2(\Z)\backslash\U(\Sigma)_{J<X}}&\leq&2^{r_2},
\end{array}
\end{equation}
where we use Theorems \ref{thsieveupperbound} and \ref{thbcfJcount} to
evaluate the numerators and the denominators of the middle terms in
the above equation, and Theorem \ref{thmprodvol} to evaluate the
product of local volumes that arise.

If the families $\fR$ are very large, then from Theorem
\ref{thsieveproj}, the inequalities in \eqref{eqproofH} and
\eqref{eqproofJ} can be replaced with equalities. Likewise, if we assume that
one of the estimates in \eqref{eqassumedte} holds, then from Theorem
\ref{thabcv}, the inequalities in \eqref{eqproofH} and
\eqref{eqproofJ} can be replaced with equalities. This concludes the
proof of Theorem \ref{thmacceptablefam}. \hfill $\square$

\subsection{Proof of Theorem \ref{cor:oddclassno}}

Since Theorem \ref{thmacceptablefam} implies Theorems
\ref{thm:avgforfields}, \ref{thm:avgforcubicfields}, and
\ref{thm:avgfororders}, it remains to prove Theorem
\ref{cor:oddclassno}. We first prove a corollary of Theorem
\ref{thm:avgforfields} and Theorem \ref{thm:avgforcubicfields} on the
proportion of maximal orders in $\fR^{r_1,r_2}_{J,\max}$ which have
odd (narrow) class number.

\begin{cor}\label{cor:posprop}
Fix an odd integer $n \geq 3$ and signature $(r_1,r_2)$. If $\fR
\subset \fR^{r_1,r_2}_{J,\max}$ corresponds to an acceptable set of binary $n$-ic forms, then:
\begin{itemize}
\item[\rm (a)] A positive proportion $($at least $1 - 2^{1 - r_1 -
  r_2})$ of maximal orders in $\fR$ have odd class number.
\item[\rm (b)] If $r_2$ is also assumed to be nonzero, then a positive
  proportion $($at least $1 - 2^{- r_2})$ of $\fR$ have odd narrow
  class number. Thus, at least a proportion of $1 - 2^{-r_2}$ of $\fR$
  have narrow class number equal to the class number.
\end{itemize}
\end{cor}

\begin{proof} 
Fix a signature $(r_1,r_2)$, and suppose for the sake of a
contradiction that a lower proportion than $1 - 2^{1-r_1-r_2}$ of
rings of integers of number fields with signature $(r_1,r_2)$ that
correspond to integral binary $n$-ic forms have odd class number. This
implies that a larger proportion than $2^{1-r_1-r_2}$ of such maximal
orders would have nontrivial $2$-torsion subgroup in their class group
and thus have $|\Cl_2| \geq 2$. Then the limsup of the mean number of
$2$-torsion elements in class groups of such maximal orders would be
strictly larger than $1 + \frac{1}{2^{n-1-r_2}}$, contradicting
Theorem \ref{thm:avgforfields}(a), Theorem
\ref{thm:avgforcubicfields}(a), Theorem
\ref{thm:avgforcubicfields}(b), or Corollary 3 in \cite{BV3}.

Now suppose for the sake of a contradiction that a lower proportion
than $1 - 2^{-r_2}$ of maximal orders in number fields of signature
$(r_1,r_2)$ in $\fR$ have odd narrow class number. We would then be
able to conclude that a larger proportion than $2^{-r_2}$ of such
maximal orders would have at least two distinct $2$-torsion elements
in its narrow class group. Then the limsup of the mean number of
$2$-torsion elements in the narrow class groups of such maximal orders
would be strictly larger than $1 + 2^{-r_2}$, contradicting Theorem
\ref{thm:avgforfields}(b). When $n = 3$, note that the narrow class
group of a complex cubic field is always equal to its class group.
\end{proof}

\begin{thm} \label{thm:indiv} Fix a signature $(r_1,r_2)$. If $\fR \subset \fR^{r_1,r_2}_{J,\max}$ is an acceptable family of rings, then we have
\begin{itemize}
\item[{\rm (a)}] $\#\bigl\{ R \in \fR : |\Disc(R)| < X \mbox{ and } 2
  \nmid |\Cl(R)|\bigr\}\gg X^{\frac{n+1}{2n-2}}$.
\item[{\rm (b)}]  If $r_2 \geq 1$, then $\#\bigl\{ R \in \fR : |\Disc(R)| < X \mbox{ and } 2 \nmid |\Cl^+(R)|\bigr\}\gg X^{\frac{n+1}{2n-2}}$.
\end{itemize}
\end{thm}
\begin{proof}
In \cite{BSW-part2}, it is proved that there exists a nonempty open
bounded set $B \subset U(\R)$, whose closure does not contain any
element having discriminant $0$, such that for any $X>0$, every
element $f\in X\cdot B\cap U(\Z)$ is {\it strongly reduced}, i.e., the
basis given in \eqref{eq:basisRf} is the unique Minkowski-reduced
basis of the ring $R_f$ corresponding to $f$. It is further shown that
if two distinct elements $f_1$ and $f_2$ of $U(\Z)$ are strongly
reduced, then the rings $R_{f_1}$ and $R_{f_2}$ corresponding to $f_1$
and $f_2$ are not isomorphic.

Let $\Sigma$ denote the collection of local specifications defining $\fR$, and let $\fR_B$ denote the family of maximal
$S_n$-orders $R$, where $R=R_f$ arises from an integral binary $n$-ic
form $f\in \U(\Sigma)\cap \R_{>0}\cdot B$. We endow this family of
binary $n$-ic forms with the natural height
\begin{equation*}
H_B(f):=\min\{X:f\in X\cdot B\},
\end{equation*}
thereby defining a height function on the family $\fR_B$ of maximal $S_n$-orders.  The
average sizes of $\Cl_2$ and $\Cl_2^+$ over the rings in $\fR_B$, ordered
by $H_B$, are bounded by $1+2^{1-r_1-r_2}$ and $1+2^{-r_2}$, respectively; the proof
for the analogous statement when rings are ordered by height $H$
adapts to this situation without change. Therefore, by the same
argument as in the proof of Corollary \ref{cor:posprop}, we see that a
positive proportion of rings in $\fR_B$ have odd class number.

Let $c>0$ be a constant such that every element in $cB$ has
discriminant bounded by $1$ in absolute value. Then every element in
$cX^{1/(2n-2)}B$ has discriminant bounded by $X$. Since we have
\begin{equation*}
\#\{\U(\Sigma)\cap cX^{1/(2n-2)}B\}\gg X^{\frac{n+1}{2n-2}},
\end{equation*}
the theorem follows.
\end{proof}

\noindent Note that the conditions required in Theorem \ref{cor:oddclassno} are indeed acceptable, so Theorem \ref{cor:oddclassno} follows directly from Theorem \ref{thm:indiv}.

\bibliography{hsvbib} \bibliographystyle{amsplain}

\end{document}